\documentclass[hidelinks,onefignum,onetabnum]{siamart251216}

\usepackage{lipsum}
\usepackage{amsfonts}
\usepackage{graphicx}
\usepackage{epstopdf}
\usepackage{algorithmic}
\ifpdf
  \DeclareGraphicsExtensions{.eps,.pdf,.png,.jpg}
\else
  \DeclareGraphicsExtensions{.eps}
\fi

\newsiamremark{remark}{Remark}
\newsiamremark{assumption}{Assumption}
\newsiamremark{problem}{Problem}
\newsiamremark{hypothesis}{Hypothesis}
\crefname{hypothesis}{Hypothesis}{Hypotheses}
\newsiamthm{claim}{Claim}
\newsiamremark{fact}{Fact}
\crefname{fact}{Fact}{Facts}

\headers{Regularized Newton-SLRA}{D. Zheng and S. Chen}

\title{A Regularized Newton-Type Method for Manifold--Affine Intersection
Problems under Intrinsic Transversality
}

\author{
Dengyu Zheng\thanks{
School of Mathematical Sciences, Key Laboratory of the Ministry of Education
for Mathematical Foundations and Applications of Digital Technology,
University of Science and Technology of China, Hefei, Anhui, China
(\email{zq\_renius919@mail.ustc.edu.cn}, \email{shxchen@ustc.edu.cn}).
}
\and
Shixiang Chen\footnotemark[1]
}

\usepackage{amsopn}
\DeclareMathOperator{\diag}{diag}

\usepackage{enumitem}

\usepackage{xcolor}
\newcommand{\blue}[1]{#1}
\colorlet{BLUE}{blue} 
\title{\blue{A Regularized Newton-Type Method for Manifold--Affine Intersection
Problems under Intrinsic Transversality}}

\usepackage{booktabs}
\usepackage{array}

\ifpdf
\hypersetup{
  pdftitle={A Regularized Newton-Type Method for Manifold-Affine Intersection Problems under Weak Intersection Conditions},
  pdfauthor={Dengyu Zheng, and Shixiang Chen}
}
\fi

\begin{document}

\maketitle

\begin{abstract}
We propose Regularized Newton-SLRA (RN-SLRA), a regularized Newton-type method
for local manifold--affine intersection problems motivated by structured
low-rank approximation. Classical Newton-SLRA achieves fast local convergence
under transversality, but its tangent-space intersection step may become
ill-defined, singular, or severely ill-conditioned when transversality fails.
RN-SLRA overcomes this difficulty by replacing the exact tangent-space
intersection step with a regularized quadratic subproblem over the affine space.
\blue{Under intrinsic transversality, RN-SLRA with regularization parameter
$\mu_k=c r_k^\rho$ converges locally with Q-order at least $1+\rho$ for
$0<\rho\le1$, including quadratically for $\rho=1$, and linearly for $\rho=0$. We also analyze an inexact variant
based on $\sigma$-quasioptimal manifold projections.  It retains the
$1+\rho$ order for $0<\rho\leq1$, and converges linearly for fixed
regularization when $q\sigma<1$, where $q$ is the local one-step
alternating-projection factor.  Experiments support the predicted quadratic behavior on an
analytically clean but nontransversal SLRA instance, where the classical Newton
tangent system is severely ill-conditioned, and show reduced projection costs
for the inexact method on large-scale Hankel problems.}

\end{abstract}



\section{Introduction}

Structured low-rank approximation (SLRA) seeks to approximate a structured
matrix \(M\) by a matrix \(M_r\) with the same structure and prescribed rank.
Typical structures include Hankel, Toeplitz, Sylvester, and more general affine
matrix structures
\cite{ChuFunderlicPlemmons2003,MARKOVSKY2008}. In contrast to classical low-rank
approximation, which admits an explicit solution by the SVD, SLRA must enforce
the affine structure and the rank constraint simultaneously. This coupling makes
the problem substantially more difficult. SLRA arises in signal processing,
system identification, and computer algebra
\cite{MARKOVSKY2008,markovsky2015behavioral,kaltofen2007sylvester}; for example,
low-rank Hankel approximation is related to exponential data fitting, frequency
estimation, and denoising, while Sylvester-structured low-rank approximation is
important in approximate polynomial GCD computations
\cite{MARKOVSKY2008,markovsky2018sparse,kaltofen2007sylvester}.

Geometrically, an affine matrix structure is represented by an affine subspace,
whereas the rank constraint is represented locally by a smooth fixed-rank
manifold. Thus SLRA naturally leads to a manifold--affine intersection problem.
Let \(\mathcal{M}_{p,q}(\mathbb{R})\) denote the space of real \(p\times q\)
matrices, equipped with the Frobenius inner product
\(
    \langle M_1,M_2\rangle=\operatorname{tr}(M_1^\top M_2)
\)
and norm
\(
    \|M\|=\sqrt{\operatorname{tr}(M^\top M)}.
\)
Denote by \(\mathcal{D}_r\) the subset of
\(\mathcal{M}_{p,q}(\mathbb{R})\) consisting of matrices of rank \(r\).
All local statements involving \(\mathcal D_r\) are understood in a neighborhood
of a rank-\(r\) point \(x^*\in E\cap\mathcal D_r\) with
\(\sigma_r(x^*)>0\); in such a neighborhood, \(\mathcal D_r\) is a smooth
fixed-rank manifold and the metric projection onto \(\mathcal D_r\) is locally
single-valued. The local SLRA problem considered in this paper can be formulated
as follows.

\begin{problem}[SLRA]
    \label{prob:SLRA}
    Consider an affine subspace \(E\subset \mathcal{M}_{p,q}(\mathbb{R})\),
    a matrix \(M\in E\), and an integer
    \(r\in \{0,1,\ldots,\min(p,q)\}\).
    Find a matrix \(M_r\in E\cap\mathcal D_r\) close to \(M\).
\end{problem}

Computing a metric projection onto \(E\cap\mathcal D_r\) is generally difficult.
In practice, one usually seeks a nearby feasible point in \(E\cap\mathcal D_r\),
which can be regarded as a local approximation to the projection onto the
manifold--affine intersection. Many algorithms have been proposed for this
purpose. Among them, the Cadzow algorithm \cite{Cadzow1988} can be viewed as an
alternating-projection method between the affine structure set and the
fixed-rank manifold. It is simple and robust, and converges locally linearly
under separability-type intersection conditions
\cite{Lewis2008,Drusvyatskiy2019}. In contrast, Newton-SLRA \cite{Schost2013} and APHL~\cite{xiao2025quadraticallyconvergentalternatingprojection} are   Newton-type variants with local
quadratic convergence under transversality.  Other approaches include structured total least squares methods
\cite{RosenParkGlick1998,ParkZhangRosen1999}, Riemannian optimization
\cite{AbsilAmodeiMeyer2014,Vandereycken2013}, alternating minimization
\cite{JainNetrapalliSanghavi2013}, variable projection
\cite{MarkovskyUsevich2014SoftwareSLRA}, gradient systems
\cite{FazziGuglielmiMarkovsky2021GradientSystem}, and algebraic methods
\cite{OttavianiSpaenlehauerSturmfels2014}.

The quadratic convergence of Newton-SLRA, however, relies on a restrictive
geometric requirement: in the SLRA setting, the affine subspace \(E\) and the
fixed-rank manifold \(\mathcal D_r\) must intersect transversally. When
transversality fails, the tangent-space intersection in the Newton-SLRA update
may be empty, and the associated linear system may become singular or severely
ill-conditioned. Thus, although alternating projections remain applicable under
weaker intersection conditions, \blue{the standard local quadratic-convergence
guarantee for Newton-SLRA is no longer available. This motivates a regularized Newton-type method that is
well defined under weaker geometry and retains higher-order local convergence
already under intrinsic transversality, equivalently clean intersection in the
present setting.}

To address this issue, we propose Regularized Newton-SLRA (RN-SLRA) for local
manifold--affine intersection problems. The method replaces the potentially
ill-defined tangent-space intersection step in Newton-SLRA by a regularized
subproblem over the affine space. The regularization parameter is chosen as
\[
    \mu_k=c\,r_k^\rho,\qquad
    r_k:=d_M(x_k)=\|x_k-P_M(x_k)\|,\qquad
    c>0,\quad \rho\in[0,1].
\]
This residual-dependent choice makes \blue{the regularization adaptive}: it
stabilizes the step when the manifold residual is not small, and for
\(\rho>0\) it becomes asymptotically weaker as the iterates approach the
intersection. \blue{Under intrinsic transversality, we prove local linear convergence
for $\rho=0$ and Q-order at least $1+\rho$ for $0<\rho\le1$, including
quadratic convergence when $\rho=1$, without classical transversality. Let \(X:=\mathcal D_r\cap E\) denote the solution manifold. Our analysis isolates the component of the error that is normal to \(X\) within the affine space \(E\). Clean intersection makes the relevant normal projector uniformly coercive on this effective error space, while motion tangent to \(X\) affects the distance to \(X\) only at second order.}
We also record a retraction interpretation of the RN-SLRA limit map under
clean intersection. \blue{The convergence analysis in Section~\ref{sec:local-convergence}
covers the full range $\rho\in[0,1]$; Appendix~\ref{app:rho-zero-retraction}
further proves that the fixed-regularization limit map is smooth and defines a
genuine retraction and, under additional smoothness, a second-order retraction.}
This observation is of independent interest for optimization problems constrained
to intersection manifolds \cite{yang2026optimizationintersectionmanifolds,chen2026retractionsalternatingprojections}.

\blue{For large-scale SLRA problems, we further introduce Inexact Regularized
Newton-SLRA (iRN-SLRA), which replaces the exact manifold projection by a
$\sigma$-quasioptimal projection} \cite{budzinskiy2025quasioptimal}.
\blue{For $0<\rho\le1$, the same intrinsic-transversality analysis yields Q-order at
least $1+\rho$, including quadratic convergence for $\rho=1$, without
classical transversality. For $\rho=0$, the local linear result is retained under the
corresponding quasioptimality contraction condition. This relaxation permits the use of cheaper approximate manifold
projections, which can reduce the per-iteration cost on large-scale
problems.} The local convergence properties of the methods
discussed above are summarized in Table~\ref{tab:comparison}.

\begin{table}[htbp]
  \caption{Local convergence rates of representative methods under different
  intersection conditions.
  \blue{The term ``linear'' refers to Q-linear convergence of the distance or residual specified in the cited result, together with R-linear convergence of the iterates. The parameter \(q\in(0,1)\) denotes the local one-step alternating-projection decrease factor defined in Proposition~\ref{prop:irn-fixed}.}}
  \label{tab:comparison}
  \centering
  \setlength{\tabcolsep}{4pt}
  \renewcommand{\arraystretch}{1.05}

  \begin{tabular}{
    >{\centering\arraybackslash}m{0.20\linewidth}
    >{\centering\arraybackslash}m{0.38\linewidth}
    >{\centering\arraybackslash}m{0.35\linewidth}
  }
    \toprule
    &
    \shortstack{Intrinsic transversality}
    &
    Transversality
    \\
    \midrule

    Cadzow (AP)
    &
    linear
    &
    linear
    \\
    \cmidrule(lr){1-3}

    Newton-SLRA
    &
    \blue{no general well-definedness guarantee}
    &
    quadratic
    \\
    \midrule

    RN-SLRA (ours)
    &
    \blue{\shortstack{
      $\rho=0$: linear;\\
      $0<\rho\le1$: order $1+\rho$;\\
      quadratic if $\rho=1$
    }}
    &
    \blue{\shortstack{
      $\rho=0$: linear;\\
      $0<\rho\le1$: order $1+\rho$;\\
      quadratic if $\rho=1$
    }}
    \\
    \cmidrule(lr){1-3}

    iRN-SLRA (ours)
    &
    \blue{\shortstack{
      $\rho=0$: linear if $\sigma q<1$;\\
      $0<\rho\le1$: order $1+\rho$;\\
      quadratic if $\rho=1$
    }}
    &
    \blue{\shortstack{
      $\rho=0$: linear if $\sigma q<1$;\\
      $0<\rho\le1$: order $1+\rho$;\\
      quadratic if $\rho=1$
    }}
    \\
    \bottomrule
  \end{tabular}
\end{table}

\blue{We evaluate RN-SLRA and iRN-SLRA on constructed SLRA instances and
low-rank Hankel approximation problems at both small and large scales. A
small-scale instance whose clean but nontransversal geometry is verified
analytically exhibits the quadratic local behavior predicted for $\rho=1$ and
also exposes the severe ill-conditioning of the classical Newton tangent
system. In addition, iRN-SLRA exhibits convergence behavior comparable to
RN-SLRA while reducing the cost of the manifold projection, with more
pronounced savings on large-scale problems.}

\section{Preliminaries}
We consider the general problem of approximating the projection of a point $x$ onto the intersection $X:=M\cap E$ of a smooth embedded manifold $M$ and an affine subspace $E$ in a Euclidean space, assuming without loss of generality that $M\neq E$; in particular, the SLRA problem is recovered when $M=\mathcal D_r$.

\subsection{Notations and Basic Facts}
We use $d_C(x)$ to denote the distance from $x$ to the set $C$, and $P_C(x)$ to represent the projection of $x$ onto the set $C$. $B(x,r)$ denotes the ball centered at $x$ with radius $r$. We use $\hat{u} = \frac{u}{\|u\|}$ to represent the unit vector of $u$. 
For any affine subspace $A$, let $A^0$ denote its underlying vector space, so that
\(
A=x+A^0 \text{ for any }x\in A.
\)
In particular, we write
\(
E=x+L \text{ for any }x\in E,
\)
where $L:=E^0$.
For an affine manifold $V$ and a point $x\in V$, we denote by $T_xV$ the affine tangent space of $V$ at $x$, and by $N_xV$ the affine normal space of $V$ at $x$. Their underlying vector spaces are denoted by $T_xV^0$ and $N_xV^0$.
$L^{\perp}$ denotes the orthogonal complement of $L$.

\begin{assumption}
  \label{ass:smooth}
  $E$ is an affine subspace and $M$ is a smooth manifold of class $C^2$. 
\end{assumption}

We first recall some intersection conditions, including the transversality condition \cite{Lewis2008,Ioffe2016Survey1,Ioffe2017Transversality}, the clean intersection condition \cite{AnderssonCarlsson2013}, the intrinsic transversality condition \cite{Drusvyatskiy2013Thesis, DrusvyatskiyIoffeLewis2015}, the separability condition \cite{NollRondepierre2016}, and the subtransversality condition \cite{KrugerLukeThao2017}.
\begin{definition}[Transversality]
  Let $\mathbb E$ be a Euclidean space. Let $M\subset \mathbb E$ be a manifold of class $C^1$ and let $E$ be an affine subspace of $\mathbb E$. We say that $E$ intersects $M$ transversally at a point $\bar x\in X$ if 
    $\operatorname{codim}(L \cap T_{\bar x} M^0) = \operatorname{codim}(L) + \operatorname{codim}(T_{\bar x} M^0). $
\end{definition}

\begin{definition}[Clean intersection]
Let $\mathbb E$ be a Euclidean space. Let $M \subset \mathbb E$ be a
manifold of class $C^p(p\geq 2)$ and let $E$ be an affine subspace of $\mathbb E$.
We say that $E$ intersects $M$ cleanly at a point $\bar x\in X$
if the set $X$ is
a $C^p$ embedded submanifold and
\(
  \blue{T_x X^0 = T_xM^0 \cap L}
\)
for all $x\in X$ sufficiently close to $\bar x$.
\end{definition}

\begin{definition}[Intrinsic Transversality]
  Let $\mathbb E$ be a Euclidean space. Let $M\subset\mathbb E$ be a manifold of class $C^1$ and let $E$ be an affine subspace of $\mathbb E$. We say that $E$ intersects $M$ intrinsically transversally at a point $\bar x\in X$ if there exists a constant $\kappa>0$ such that for all $x\in E\setminus M$ and $y\in M\setminus E$ sufficiently close to $\bar x$,
  \[
  \max \left\{ d\left( \widehat{x - y}, N_y M^0 \right), \; d\left( \widehat{x - y}, N_x E^0 \right) \right\} \ge \kappa.
  \]
\end{definition}

\begin{definition}[Separability]
Let $\mathbb E$ be a Euclidean space. Let $M\subset\mathbb E$ be a manifold of class $C^1$ and let $E$ be an affine subspace of $\mathbb E$. We say that $E$ intersects $M$ separably at a point $\bar{x} \in X$ if there exists an angle $\alpha > 0$ such that for any point $z \in E \setminus M$ sufficiently close to $\bar{x}$, and any points $x \in P_M(z) \setminus E$ and $z' \in P_E(x)$, the angle between the vectors $z - x$ and $z' - x$ is at least $\alpha$.
\end{definition}

The notion of separability was  introduced in \cite{NollRondepierre2016} under the term 0-separability. This property is particularly suitable for the analysis of alternating projections and inexact alternating projection methods \cite{Drusvyatskiy2019}. Intrinsic transversality is the principal geometric assumption in all
of our local convergence results, including the higher-order results.

\begin{assumption}
  \label{ass:intrinsic}
  $E$ intersects $M$ intrinsically transversally at $\bar x$.
\end{assumption}

Under Assumption~\ref{ass:smooth}, intrinsic transversality is equivalent to separability in either order; see Lemma~\ref{lem:equiv-sep-intrinsic} below.
Although separability is not symmetric in general, in the present $C^2$ manifold--affine setting the two directional versions are equivalent.
More precisely, $E$ intersects $M$ separably at $\bar x$ if and only if $M$ intersects $E$ separably at $\bar x$, and these are both equivalent to intrinsic transversality.

\begin{lemma}
  \label{lem:equiv-sep-intrinsic}
Let $\mathbb E$ be a Euclidean space, let $M \subset \mathbb E$ be a $C^2$ manifold, and let $E \subset \mathbb E$ be an affine subspace. Then the following are equivalent at $\bar x \in E \cap M$:
\begin{enumerate}
    \item[(1)] $E$ intersects $M$ separably at $\bar x$;
    \item[(2)] $M$ intersects $E$ separably at $\bar x$;
    \item[(3)] $E$ intersects $M$ intrinsically transversally at $\bar x$.
\end{enumerate}
\end{lemma}

\begin{proof}
Write $E=\bar x+L$, so $\blue{N_x E^0=L^\perp}$ for all $x\in E$.

\smallskip
\noindent
$(1)\Rightarrow(2)$.
Assume that $E$ intersects $M$ separably at $\bar x$. 
After possibly shrinking $U$, we may assume that $P_E(M\cap U)\subset U$ and $P_M(E\cap U)\subset U$ by continuity of $P_E$ at $\bar x$ and local single-valuedness and continuity of $P_M$ near $\bar x$.

Take
\(
y\in (M\cap U)\setminus E,\ x=P_E(y)\in U\setminus M,\ y^+\in P_M(x)\cap U,
\)
and set $x^+:=P_E(y^+)$. If $y^+\in E$, then
\(
y^+-x\in L,\ y-x\in L^\perp,
\)
hence
\(
\angle(y-x,y^+-x)=\frac{\pi}{2}\ge \alpha.
\)
Assume now $y^+\notin E$. Since $x\in E$, $y^+\in P_M(x)\setminus E$, and $x^+=P_E(y^+)$, separability gives
\(
\angle(x-y^+,x^+-y^+)\ge \alpha.
\)
Moreover,
\(x^+-y^+ = P_{L^\perp}(x-y^+)\) and \(x-y\in L^\perp\). By Cauchy's
inequality, among all nonzero vectors \(w\in L^\perp\), the angle between
\(x-y^+\) and \(w\) is minimized when \(w\) is parallel to
\(P_{L^\perp}(x-y^+)\). Hence
\(
    \angle(x-y^+,x-y)
    \ge \angle(x-y^+,x^+-y^+)
    \ge \alpha .
\)
Equivalently,
\(
\angle(y-x,y^+-x)\ge \alpha.
\)
Thus $M$ intersects $E$ separably at $\bar x$.

\smallskip
\noindent
$(2)\Rightarrow(3)$.
Assume that $M$ intersects $E$ separably at $\bar x$, but intrinsic transversality fails. Then there exist
\(
a_k\in E\setminus M,\ y_k\in M\setminus E,\ a_k\to\bar x,\ y_k\to\bar x,
\)
such that, with $w_k:=\widehat{y_k-a_k}$,
\[
\max\{d(w_k,L^\perp),\,d(w_k,N_{y_k}M^0)\}\to 0.
\]
Set
\(
x_k:=P_E(y_k),\ t_k:=\|y_k-x_k\|,\ u_k:=\frac{y_k-x_k}{\|y_k-x_k\|}.
\)
Then $u_k\in L^\perp$ and
\(
P_{L^\perp}(y_k-a_k)=y_k-x_k,
\
P_{L^\perp}w_k=\frac{t_k}{\|y_k-a_k\|}u_k.
\)
Since $d(w_k,L^\perp)\to0$, we have $\|P_{L^\perp}w_k\|\to1$, hence
\(
u_k=\frac{P_{L^\perp}w_k}{\|P_{L^\perp}w_k\|}\to w_k.
\)
Choose $v_k\in N_{y_k}M^0$ with $\|v_k-w_k\|=d(w_k,N_{y_k}M^0)\to0$. Then
\(
\|u_k-v_k\|\to 0. 
\)
Define
\(
z_k:=y_k-t_kv_k.
\)
Since $M$ is a $C^2$ manifold, for all large $k$, $z_k$ lies in a tubular neighborhood of $M$ and
\(
P_M(z_k)=y_k.
\)
Moreover, $P_M$ is locally Lipschitz there. As
\(
x_k=y_k-t_ku_k,
\
z_k-x_k=t_k(u_k-v_k),
\)
we obtain
\[
\|z_k-x_k\|=t_k\|u_k-v_k\|=o(t_k),
\qquad
\|z_k-y_k\|=t_k\|v_k\|.
\]
Since $\|v_k-w_k\|\to 0, \|w_k\|=1$, we have $\|z_k-y_k\|=(1+o(1))t_k.$ We claim that for all large $k$,  $x_k\notin M$. Otherwise,
passing to a subsequence if necessary, assume $x_k\in M$. Then $d_M(z_k)\leq \|z_k-x_k\|=o(t_k)$, whereas $P_M(z_k)=y_k$ implies $d_M(z_k)=t_k\|v_k\|=(1+o(1))t_k$, a contradiction.

Now let $y_k^+:=P_M(x_k)$. By local Lipschitz continuity of $P_M$,
\[
\|y_k^+-y_k\|
=
\|P_M(x_k)-P_M(z_k)\|
\le C\|x_k-z_k\|
=o(t_k).
\]
Since $\|y_k-x_k\|=t_k$, it follows that
\(
\|(y_k^+-x_k)-(y_k-x_k)\|=o(t_k),
\)
and therefore
\(
\angle(y_k-x_k,y_k^+-x_k)\to0.
\)
But for all large $k$,
\(
y_k\in M\setminus E,\ x_k=P_E(y_k)\in E\setminus M,\ y_k^+\in P_M(x_k),
\)
so separability of $M$ relative to $E$ yields some $\alpha_0>0$ such that
\(
\angle(y_k-x_k,y_k^+-x_k)\ge \alpha_0,
\)
a contradiction. Hence $E$ and $M$ are intrinsically transversal at $\bar x$.

\smallskip
\noindent
\((3)\Rightarrow(1)\). This follows from the fact that intrinsic transversality implies separability; see \cite[Proposition~2]{NollRondepierre2016}. 
Therefore $(1)$, $(2)$, and $(3)$ are equivalent.
\end{proof}

\begin{definition}[Subtransversality]
Let $\mathbb E$ be a Euclidean space. Let $M\subset\mathbb E$ be a manifold of class $C^1$ and let $E$ be an affine subspace of $\mathbb E$. We say that $E$ intersects $M$ subtransversally at a point $\bar{x} \in X$ if there exist constants $\tau > 0$ and a neighborhood $U$ of $\bar{x}$ such that for all $w \in U$,
\[
d(w, X) \le \tau \left( d(w, E) + d(w, M) \right).
\]
\end{definition}

Under the present \(C^2\) embedded manifold--affine setting and with the
definitions used here, the following relationships hold locally, rather than as
unconditional statements for arbitrary sets: transversality implies clean
intersection; 
clean intersection is equivalent to intrinsic transversality\cite[Theorem~5.1]{yang2026optimizationintersectionmanifolds};
intrinsic transversality implies subtransversality; and intrinsic transversality
is equivalent to separability in either order; see
\cite{DrusvyatskiyIoffeLewis2015,NollRondepierre2016,yang2026optimizationintersectionmanifolds}.

\subsection{Cadzow Algorithm}
The alternating projection (AP) method has a long history \cite{Schwarz1869,vonNeumann1950}. Recent studies have analyzed AP and
inexact AP methods for various sets \cite{BauschkeBorwein1993,LewisLukeMalick2009,KrugerThao2016}, together with their local linear convergence properties \cite{LewisLukeMalick2009,NollRondepierre2016,LukeTeboulleThao2020,Drusvyatskiy2019}.
In the context of structured low-rank approximation, the Cadzow algorithm \cite{Cadzow1988,Gillard2010} can be viewed as an alternating projection method between the affine space $E$ and the low-rank manifold $M$.
More precisely, starting from $x_k\in E$, one step of the Cadzow iteration is
given by
\begin{align}
  x_{k+1}=P_E P_M(x_k).
\end{align}

In the local rank-\(r\) neighborhood considered in this paper, it is well known that  the projection
onto \(\mathcal D_r\) can be  computed by the truncated SVD.
\begin{lemma}[Eckart-Young Theorem]
  \label{thm:svd}
  Let $x\in\mathcal{M}_{p,q}(\mathbb{R})$ be a matrix, and let $r$ be a positive integer not greater than $\min(p,q)$. Let $x=USV^\top$ be its SVD, where $S=\diag(\sigma_1,\cdots,\sigma_{\min(p,q)})$ with $\sigma_1\geq \cdots \geq \sigma_{\min(p,q)}\geq 0$. If \(x\) is sufficiently close to a rank-\(r\) point \(x^*\) with \(\sigma_r(x^*)>0\), then a projection $P_{\mathcal{D}_r}(x)$ is given by 
    $$ P_{\mathcal D_r}(x)=U\tilde{S}V^\top,\ \ \tilde S=\diag(\sigma_1,\cdots,\sigma_r,0,\cdots,0).  $$
\end{lemma}
Then the Cadzow algorithm can be formulated as:
\begin{equation}
\begin{aligned}
    &U,S,V \leftarrow \text{SVD}(x_k),\qquad y_k=U_rS_rV_r^\top, \\
    &x_{k+1}=x_k+\sum_{i=1}^d \langle y_k-x_k,E_i \rangle E_i.
\end{aligned}
\end{equation}
where $U_r$ denotes the matrix consisting of the first $r$ columns of $U$, $V_r$ denotes the matrix consisting of the first $r$ columns of $V$, $S_r$ denotes the $r\times r$ top-left submatrix of $S$, $d$ is the dimension of $E$, and $\{E_i\}_{i=1}^d$ denotes an orthonormal basis of $L$. 
The update formula for $x_{k+1}$ is obtained by 
  $$x_{k+1}=P_E(y_k)=x_k+P_L(y_k-x_k).$$

If $E$ intersects $M$ transversally at $\bar x\in X$ and the initial point is sufficiently close to $\bar x$, the Cadzow algorithm   converges locally linearly to a point $x_{\infty}\in X$ \cite{Lewis2008}. In fact, local linear convergence can already be guaranteed when $E$
intersects $M$ separably \cite{Drusvyatskiy2019}.

\subsection{Newton-SLRA Algorithm}
The Newton-SLRA algorithm is motivated by Newton's method. In each iteration of the classical Newton-SLRA, given the current iterate \(x_k \in E\), one first computes its projection onto the manifold \(y_k = P_{M}(x_k)\), and then solves a linear least-squares problem:
\(
x_{k+1} = P_{E \cap T_{y_k}M}(x_k),
\)
i.e., the closest point to \(x_k\) in the intersection of \(E\) and the tangent space \(T_{y_k}M\). Under the transversality condition, this algorithm converges locally quadratically. When $M=\mathcal D_r$, the algorithm can be formulated as:
\begin{equation}
\label{eq:newton-slra}
\begin{aligned}
    &U,S,V\leftarrow \mathrm{SVD}(x_k),\qquad y_k = U_r S_r V_r^\top,\\
    &N_{ij}\leftarrow u_i v_j^\top,\quad i=r+1,\dots,p,\ \ j=r+1,\dots,q,\\
    &A\leftarrow \big(\langle N_{ij},E_\ell\rangle\big)_{(i,j),\ell}, \qquad b\leftarrow \big(\langle N_{ij},\,y_k-x_k\rangle\big)_{(i,j)},\\
    &x_{k+1}=x_k+\displaystyle\sum_{\ell=1}^d (A^\dagger b)_\ell E_\ell.
\end{aligned}
\end{equation}

However, when transversality fails, the matrix \(A\) in~\eqref{eq:newton-slra} may lose rank at the solution. In the fixed-rank SLRA setting, full row rank of the tangent-space linear system is equivalent to transversality of \(E\) and \(\mathcal D_r\) at the intersection point. Hence, near a nontransversal solution, the associated least-squares problem may become singular or severely ill-conditioned.

Another issue is that the well-definedness of Newton-SLRA relies on the non-emptiness of \( E \cap T_{y_k} M \). Under transversality, there exists a positive lower bound \( \alpha \) for the angle between \( E \) and \( T_{y_k} M \) in some neighborhood of \( \bar{x} \), but this cannot be guaranteed under intrinsic transversality, as \( E \cap T_{y_k} M \) may be empty.

These issues motivate the regularized modification introduced below.

\section{Regularized Newton-SLRA}
We introduce a regularization term and define the next iterate as:
\begin{equation}
\label{eq:rnslra}
\begin{aligned}
    &y_k=P_M(x_k),\\
    &x_{k+1} = \arg\min_{x \in E} \left\{ \dfrac{1}{2} \| x - y_k \|^2 + \dfrac{1}{2\mu_k} \| P_{\blue{N_{y_k}M^0}}(x - y_k) \|^2 \right\}, \quad \mu_k > 0.
\end{aligned}
\end{equation}
where \(P_{N_{y_k}M}\) is the orthogonal projector onto the normal space \(N_{y_k}M\), and \(\mu_k\) is a regularization parameter. In the algorithm studied below, $\mu_k$ is chosen adaptively according to: 
  $$ \mu_k = c \cdot r(x_k)^\rho,\quad r(x_k)=\|x_k-y_k\|=d_M(x_k). $$
\blue{Here $c>0$ is fixed and $\rho\in[0,1]$ controls how rapidly the
regularization vanishes with the manifold residual.}
This choice of the regularization parameter is inspired by the residual-dependent regularization strategy used in the Inexact Regularized Proximal Newton method (IRPN) \cite{yue2018}, and is also related to earlier regularized Newton methods for convex minimization with singular solutions \cite{LiFukushimaQiYamashita2004} as well as improved proximal Newton implementations such as newGLMNET \cite{YuanHoLin2012}.

Since \(E\) is affine, the objective function in \eqref{eq:rnslra} for solving $x_{k+1}$ is a strongly convex quadratic problem with a unique solution, which can be obtained by solving a linear system.
Thus, the well-definedness of the RN-SLRA step does not require the transversality condition between $E$ and $M$. 

\subsection{Derivation of the Linear System}

Let \(\{E_1,\dots,E_d\}\) be an orthonormal basis of \(L\), set
\(b=x_k-y_k\), and write \(N=N_{y_k}M\). Choose an orthonormal basis
\(\{N_j\}_{j=1}^s\) of \(N\). Since \(E=x_k+L\), any
\(x\in E\) can be written as
\(x=y_k+b+\sum_{i=1}^d\alpha_iE_i\). Substituting this expression into
\eqref{eq:rnslra} and using the first-order optimality condition gives
\begin{align}
  \label{eq:sys}
  \left(I_d+\frac{1}{\mu_k}A^\top A\right)\alpha
  =
  -\beta-\frac{1}{\mu_k}A^\top \eta,
\end{align}
where
\(
A_{ji}=\langle N_j,E_i\rangle,
\eta_j=\langle N_j,b\rangle,\)
and
\(
\beta_i=\langle b,E_i\rangle.
\)
Since \(A^\top A\succeq0\) and \(\mu_k>0\), the coefficient matrix is positive
definite. The next iterate is
\(x_{k+1}=x_k+\sum_{i=1}^d\alpha_iE_i\).

For the fixed-rank manifold $M=\mathcal D_r$, this yields Algorithm~\ref{alg:RN-SLRA}.
\Crefname{ALC@unique}{Line}{Lines}

\begin{algorithm}[t]
\caption{Regularized Newton-SLRA (RN-SLRA)}
\label{alg:RN-SLRA}
\begin{algorithmic}[1]
\STATE \textbf{Input:} Initial point $x_0\in E$, parameters $c>0$ and $0\le\rho\le 1$
\STATE Write $E=x_0+L$, where $L=E^0$, and fix an orthonormal basis $\{e_1,\dots,e_d\}$ of $L$
\FOR{$k=0,1,2,\dots$}
    \STATE Compute an SVD $x_k=U\Sigma V^\top$ and set
    \(
        y_k:=U_r\Sigma_rV_r^\top\in P_{\mathcal D_r}(x_k).
    \)
    \STATE Set $r(x_k):=\|x_k-y_k\|$; \blue{\textbf{if} $r(x_k)=0$, \textbf{stop}; otherwise} set $\mu_k:=c\,r(x_k)^\rho$ and $b_k:=x_k-y_k$.
    \STATE For $a=r+1,\dots,m$ and $b=r+1,\dots,n$, define $N_{ab}:=u_av_b^\top$
    \STATE Form \(A_k\), \(\beta_k\), and \(\eta_k\) by
        \((A_k)_{(a,b),i}=\langle N_{ab},e_i\rangle\),
        \((\beta_k)_i=\langle b_k,e_i\rangle\), and
        \((\eta_k)_{(a,b)}=\langle N_{ab},b_k\rangle\).
    \STATE Solve \((I_d+\mu_k^{-1}A_k^\top A_k)\alpha^{(k)}=-\beta_k-\mu_k^{-1}A_k^\top\eta_k\).
    \STATE Set $x_{k+1}:=x_k+\sum_{i=1}^d \alpha_i^{(k)}e_i$
\ENDFOR
\end{algorithmic}
\end{algorithm}

\subsection{Limiting Interpretation of the Regularized Step}

Assume that \(E\) and \(M\) intersect transversally at \(\bar x\), so that
\(E\cap T_{y_k}M\) is nonempty for \(y_k\) sufficiently close to \(\bar x\).
The role of \(\mu_k\) can be understood from two limiting regimes.

If \(\mu_k\to\infty\), the penalty term vanishes and the subproblem reduces to
\(\min_{x\in E}\|x-y_k\|^2\). Hence the step becomes
\(p_k=P_E(y_k)\), namely the alternating projection step. 
If $\mu_k \to 0$, the penalty term enforces
$P_{\blue{ N_{y_k}M^0}}(x-y_k)=0$, or equivalently $x\in T_{y_k}M$.
Thus the limiting problem is
\(
    \min_{x\in E\cap T_{y_k}M} \|x-y_k\|^2 .
\)
Since $y_k=P_M(x_k)$, we have $x_k-y_k\in \blue{N_{y_k}M^0}$. Hence, for every
$x\in T_{y_k}M$,
\(
    \|x-x_k\|^2
    =
    \|x-y_k\|^2+\|x_k-y_k\|^2 .
\)
Therefore, minimizing the distance to $y_k$ over $E\cap T_{y_k}M$ is equivalent
to minimizing the distance to $x_k$ over the same set. The limiting step is the
Newton--SLRA point
\(
    q_k=P_{E\cap T_{y_k}M}(x_k),
\)
provided that $E\cap T_{y_k}M$ is nonempty.

Thus, for \(\mu_k\in(0,\infty)\), the regularized step interpolates between
alternating projections and Newton-SLRA: smaller \(\mu_k\) gives a more
Newton-like step, while larger \(\mu_k\) gives a more projection-like step.
This interpretation relies on the local nonemptiness of \(E\cap T_{y_k}M\),
which follows from transversality but may fail under
Assumption~\ref{ass:intrinsic}; it is used only as intuition and not in the
convergence analysis below.

\subsection{Inexact Regularized Newton-SLRA}
We replace \(y_k=P_M(x_k)\) in \eqref{eq:rnslra} by a quasioptimal projection and obtain iRN-SLRA:
\begin{equation}
  \label{eq:iRN-SLRA}
\begin{aligned}
    &y_k\in M \ \text{such that}\ \|x_k-y_k\|\le \sigma\, d_M(x_k),
    \quad \sigma\ge 1, \\
    &x_{k+1}=\displaystyle \arg\min_{x\in E}
    \left\{
      \frac12\|x-y_k\|^2
      +\frac{1}{2\mu_k}\bigl\|P_{\blue{N_{y_k}M^0}}(x-y_k)\bigr\|^2
    \right\},
    \quad \mu_k>0.
\end{aligned}
\end{equation}
When \(\sigma=1\), iRN-SLRA reduces to RN-SLRA.
For \(\sigma>1\), \(y_k\) is a \(\sigma\)-quasioptimal projection of \(x_k\)
onto \(M\), in the sense of \cite{budzinskiy2025quasioptimal}.
This quasioptimal model is well suited to low-rank approximation: truncated
SVD gives the exact projection with \(\sigma=1\), while randomized SVD provides a practical approximate rank-\(r\) projection,
with residual error bounds available in expectation or with high probability
under standard randomized low-rank approximation results
\cite{halko2011finding}; partial SVD can also be
used in practice \cite{larsen1998lanczos}. Compared with a full SVD, these
approaches can significantly reduce the computational cost, especially for
large-scale problems.

For an \(m\times n\) rank-\(r\) SLRA problem with affine dimension \(d\), one RN-SLRA iteration consists of:
(i) computing the rank-\(r\) projection, implemented by a truncated or full SVD;
(ii) forming the normal-coordinate matrix \(A_k\in\mathbb R^{(m-r)(n-r)\times d}\);
and (iii) solving a \(d\times d\) symmetric positive definite linear system.
Thus the dominant cost is typically the rank-\(r\) projection and the construction of \(A_k\), while the \(d\times d\) solve is moderate when \(d\) is small or structured.
For Hankel/Toeplitz structures, the affine basis and matrix-vector products can often be represented implicitly, and randomized or partial SVD can reduce the projection cost in large-scale settings.

\color{black}
\section{\texorpdfstring{\textcolor{black}{Local Convergence Analysis}}{Local Convergence Analysis}}
\label{sec:local-convergence}
\leavevmode\blue{In this section, we establish local convergence of RN-SLRA over the full
range $\rho\in[0,1]$ directly from the clean-intersection geometry implied by
Assumption~\ref{ass:intrinsic}. The normal projector is uniformly coercive on
the effective error space orthogonal to the solution manifold, whereas motion
tangent to the solution manifold contributes to the distance error only at
second order. For iRN-SLRA, the same mechanism yields higher-order convergence
when $\rho>0$, while a short alternating-projection argument treats fixed
regularization. Throughout, the iteration is terminated whenever the current
residual vanishes.}

\color{black}
\subsection{\texorpdfstring{\blue{Local Convergence of RN-SLRA under Intrinsic Transversality}}{Local Convergence of RN-SLRA under Intrinsic Transversality}}
\leavevmode\blue{Under Assumptions~\ref{ass:smooth} and~\ref{ass:intrinsic}, the
intersection is locally clean, so $X$ is a $C^2$ embedded submanifold and
$T_zX^0=L\cap T_zM^0$ for $z\in X$ near $\bar x$.}
\begin{lemma}
	\label{lem:dist-y}
	Suppose that Assumptions~\ref{ass:smooth} and \ref{ass:intrinsic} hold. There exist constants $\delta_1>0$ and $\kappa>0$ such that, for every $y\in M\cap B(\bar x,\delta_1)$,
	\(
	d_{X}(y)\leq \kappa\, d_E(y).
	\)
	\end{lemma}

	\begin{lemma}
	\label{lem:dist-x}
	Suppose that Assumptions~\ref{ass:smooth} and \ref{ass:intrinsic} hold. There exist constants $\delta_2>0$ and $\kappa'>0$ such that, for every $x\in E\cap B(\bar x,\delta_2)$,
	\(
	d_{X}(x)\leq \kappa' d_M(x).
	\)
	\end{lemma}

  Both of the above lemmas can be derived from subtransversality, which itself follows from intrinsic transversality.
  These two lemmas show that there exists a neighborhood $U$ of $\bar x$ such that the distance from a point in $E\cap U$ to $M$, or from a point in $M\cap U$ to $E$, is of the same order as its distance to $X$.

\begin{proposition}[Second-order estimates near a $C^2$ manifold]
\label{prop:normal-second-order}
Suppose that Assumption~\ref{ass:smooth} holds. \blue{Then there exist a neighborhood
$U$ of $\bar x$, a radius $\delta>0$, and constants $\eta,G,G'>0$ such that
the following statements hold:}
\begin{enumerate}[label=(\arabic*)]
    \item \blue{for all $y,z\in M\cap U$ satisfying $\|z-y\|<\delta$,}
    \begin{equation}
    \|P_{N_yM^0}(z-y)\|
    \le \eta \|z-y\|^2;
    \label{eq:normal-quadratic-prop}
    \end{equation}

    \item \blue{for all $y\in M\cap U$ and $x\in U$ satisfying
    $\|x-y\|<\delta$,}
    \begin{equation}
    d_M(x)
    \le \|P_{N_yM^0}(x-y)\|
    +G\|x-y\|^2;
    \label{eq:distance-normal-upper-prop}
    \end{equation}

    \item \blue{for all $y\in M\cap U$ and $x\in U$ satisfying
    $\|x-y\|<\delta$,}
    \begin{equation}
    \|P_{N_yM^0}(x-y)\|
    \le d_M(x)
    +G'\|x-y\|^2.
    \label{eq:distance-normal-lower-prop}
    \end{equation}
\end{enumerate}
\end{proposition}

\begin{proof}
By the local $C^2$ geometry of $M$, for each $y\in M$ near $\bar x$ and each $u\in T_yM^0$ small, points of $M$ near $y$ can be represented as $y+u+v_y(u)$ with $v_y(u)\in N_yM^0$. Here $v_y(0)=0$ and $Dv_y(0)=0$; equivalently, after shrinking the neighborhood,
\(
\|v_y(u)\|\le c \|u\|^2
\)
uniformly for $y$ near $\bar x$; see \cite[Lemma~20]{AbsilMalick2012}. Therefore, if $z\in M$ is close to $y$ and $u=P_{T_yM^0}(z-y)$, then
\(
z-y=u+v_y(u),
\)
so
\(
P_{N_yM^0}(z-y)=v_y(u),
\)
which gives \eqref{eq:normal-quadratic-prop}.

Next, for $x$ close to $y$, let $u=P_{T_yM^0}(x-y)$ and set $w:=y+u+v_y(u)\in M$. Then
\(
x-w=P_{N_yM^0}(x-y)-v_y(u),
\) 
and hence
\[
d_M(x)\le \|x-w\|
\le \|P_{N_yM^0}(x-y)\|+c\|u\|^2
\le \|P_{N_yM^0}(x-y)\|+c\|x-y\|^2,
\]
which proves \eqref{eq:distance-normal-upper-prop}.

Finally, let $p:=P_M(x)$. Since the metric projection onto $M$ is single-valued near $M$ \cite[Lemma~2.1]{Lewis2008}, applying \eqref{eq:normal-quadratic-prop} to $y,p\in M$ yields
\(
\|P_{N_yM^0}(p-y)\|\le \eta\|p-y\|^2.
\) 
Thus
\[
\|P_{N_yM^0}(x-y)\|
\le \|P_{N_yM^0}(x-p)\|+\|P_{N_yM^0}(p-y)\|
\le d_M(x)+\eta\|p-y\|^2.
\]
Since $\|p-y\|\le \|p-x\|+\|x-y\|\le 2\|x-y\|$, \eqref{eq:distance-normal-lower-prop} follows.
\end{proof}

\begin{lemma}[Step-size estimate]
  \label{lem:step-size}
Suppose that Assumptions~\ref{ass:smooth} and \ref{ass:intrinsic} hold.
There exist $\delta_3>0$ and $C_1>0$ such that, for all
iterates $x_k\in E\cap B(\bar x,\delta_3)$,
\[
\|x_{k+1}-x_k\|\le C_1\,r(x_k).
\]
\end{lemma}

\begin{proof}
Let
\(
r_k:=r(x_k)=d_M(x_k),\ y_k:=P_M(x_k).
\)
Since $x_k\in E$, we have
\(
d_E(y_k)\le \|y_k-x_k\|=r_k.
\)
Hence, by Lemma \ref{lem:dist-y}, after shrinking the neighborhood if necessary, there exists $\kappa>0$ such that
\(
d_X(y_k)\le \kappa\, d_E(y_k)\le \kappa r_k .
\)
Choose $z_k\in P_X(y_k)$. Then
\begin{align}
\|z_k-y_k\|=d_X(y_k)\le \kappa r_k .
\label{eq:5}
\end{align}

Now define the regularized objective
\[
\Phi_k(x):=\frac12\|x-y_k\|^2+\frac{1}{2\mu_k}\|P_{\blue{N_{y_k}M^0}}(x-y_k)\|^2,
\qquad x\in E.
\]
By construction, $x_{k+1}$ minimizes $\Phi_k$ over $E$, so
\begin{align}
\Phi_k(x_{k+1})\le \Phi_k(z_k).
\label{eq:6}
\end{align}

Next we estimate the normal component of $z_k-y_k$. 
Since $M$ is a $C^2$ manifold near $\bar x$ (by Assumption \ref{ass:smooth}), applying \eqref{eq:normal-quadratic-prop} with $y=y_k$ and $z=z_k$ and using \eqref{eq:5}, one has
\begin{align}
\|P_{\blue{N_{y_k}M^0}}(z_k-y_k)\|
\le \eta  \|z_k-y_k\|^2
\le \eta  \kappa^2 r_k^2 .
\label{eq:8}
\end{align}

Therefore,
\begin{align*}
\Phi_k(z_k)
= 
\frac12\|z_k-y_k\|^2
+\frac{1}{2\mu_k}\|P_{\blue{N_{y_k}M^0}}(z_k-y_k)\|^2 
\le
\frac12\kappa^2 r_k^2
+\frac{\eta ^2\kappa^4}{2\mu_k} r_k^4 .
\label{eq:9}
\end{align*}
Using $\mu_k=c\,r_k^\rho$ with $c>0$ and $0\le\rho\le 1$, we get
\begin{align}
\Phi_k(z_k)
\le
\frac12\kappa^2 r_k^2
+\frac{\eta ^2\kappa^4}{2c} r_k^{4-\rho}.
\label{eq:10}
\end{align}
Since $4-\rho\ge 3>2$, after shrinking the neighborhood further we may assume $r_k\le 1$, and hence
\(
\Phi_k(z_k)\le C\, r_k^2
\)
for some constant $C>0$ independent of $k$.

Since the second term in $\Phi_k$ is nonnegative, this bound and \eqref{eq:6} imply
\(
\frac12\|x_{k+1}-y_k\|^2\le \Phi_k(x_{k+1})\le C r_k^2 ,
\) 
and therefore
\begin{align}
\|x_{k+1}-y_k\|\le \sqrt{2C}\, r_k .
\label{eq:14}
\end{align}
Finally, by the triangle inequality,
\begin{align*}
\|x_{k+1}-x_k\|
\le \|x_{k+1}-y_k\|+\|y_k-x_k\| 
\le \sqrt{2C}\, r_k+r_k
= (1+\sqrt{2C})\, r_k .
\end{align*}
Thus the conclusion holds with
\(
C_1:=1+\sqrt{2C}.
\)
This proves the lemma.
\end{proof}

\blue{For $z\in X$ near $\bar x$, set
$W_z:=L\ominus T_zX^0$, $Q_z:=P_{N_zM^0}$, and $A_z:=P_LQ_z|_L$.}

\begin{lemma}[\blue{Uniform coercivity on the effective error space}]
\label{lem:effective-coercivity}
\blue{Suppose that Assumptions~\ref{ass:smooth} and~\ref{ass:intrinsic} hold.
After possibly shrinking the neighborhood of $\bar x$, there exists $\lambda_*>0$
such that
\[
  \langle A_z h,h\rangle=\|Q_zh\|^2\ge \lambda_*\|h\|^2,
  \qquad h\in W_z,
\]
for every $z\in X$ sufficiently close to $\bar x$.}
\end{lemma}

\begin{proof}
\blue{
Assumptions~\ref{ass:smooth} and~\ref{ass:intrinsic} imply local clean
intersection. Hence, after possibly shrinking the neighborhood of $\bar x$,
$X$ is a $C^2$ embedded submanifold and
\(
T_zX^0=L\cap T_zM^0
\)
for every $z\in X$ in this neighborhood. Recall that
$Q_z=P_{N_zM^0}$ and $A_z=P_LQ_z|_L$. For $h\in L$,
\[
\langle A_zh,h\rangle
=\langle Q_zh,h\rangle
=\|Q_zh\|^2,
\]
so $A_z$ is self-adjoint and positive semidefinite on $L$. Moreover,
\[
\ker A_z
=\{h\in L:Q_zh=0\}
=L\cap T_zM^0
=T_zX^0.
\]
Thus $A_z$ is positive definite on
$W_z=L\ominus T_zX^0$ for each fixed $z$.

It remains to show that the lower bound is uniform in $z$.
Suppose otherwise. Then there exist $z_j\in X$ with
$z_j\to\bar x$ and unit vectors $h_j\in W_{z_j}$ such that
\(
\|Q_{z_j}h_j\|\to0.
\)
Passing to a subsequence, we may assume that $h_j\to h$.
Since $h_j\in L$ and $\|h_j\|=1$, we have $h\in L$ and
$\|h\|=1$. As $X$ is a $C^2$ embedded submanifold, its tangent
spaces vary continuously; since $h_j\perp T_{z_j}X^0$, it follows that
\(
h\perp T_{\bar x}X^0.
\)

Likewise, continuity of the normal projector $z\mapsto Q_z$ gives
\[
Q_{\bar x}h
=\lim_{j\to\infty}Q_{z_j}h_j
=0.
\]
Hence $h\in T_{\bar x}M^0$. Since also $h\in L$, clean intersection yields
\(
h\in L\cap T_{\bar x}M^0=T_{\bar x}X^0,
\)
contradicting $h\perp T_{\bar x}X^0$ and $\|h\|=1$.
Therefore there exists $\lambda_*>0$, independent of $z$ sufficiently
close to $\bar x$, such that
\[
\|Q_zh\|^2\ge \lambda_*\|h\|^2,
\qquad h\in W_z.
\]
Together with
$\langle A_zh,h\rangle=\|Q_zh\|^2$, this proves the claim.
}
\end{proof}

\blue{For $\mu>0$ and $z\in X$ near $\bar x$, define the effective-error
operator}
\begin{equation}
\blue{S_z(\mu):=\mu(A_z+\mu I)^{-1}(I-A_z)\big|_{W_z}.}
\label{eq:effective-error-operator}
\end{equation}
\blue{Since $A_z|_{W_z}$ is self-adjoint and
$\lambda_*I\preceq A_z|_{W_z}\preceq I$, spectral calculus gives the uniform
bound}
\begin{equation}
\blue{\|S_z(\mu)\|\le\vartheta(\mu)
:=\max_{\lambda\in[\lambda_*,1]}
\frac{\mu(1-\lambda)}{\mu+\lambda}
=\frac{\mu(1-\lambda_*)}{\mu+\lambda_*}<1.}
\label{eq:effective-error-contraction}
\end{equation}

\begin{theorem}[\blue{Local convergence of RN-SLRA}]
\label{thm:superlinear-trans}
\blue{Suppose that Assumptions~\ref{ass:smooth} and~\ref{ass:intrinsic} hold.
Let RN-SLRA use \(\mu_k=c r_k^\rho\), where \(r_k:=d_M(x_k)\), \(c>0\),
and \(0\le\rho\le1\). Then there exist neighborhoods \(U_0\subset U\) of
\(\bar x\) and a constant \(C>0\) such that, for every
\(x_0\in E\cap U_0\), the iterates remain in \(U\) and converge to a point
\(x_\infty\in X\). Moreover, for every nonterminal iteration,}
\begin{equation}
\blue{d_X(x_{k+1})\le
\vartheta(\mu_k)d_X(x_k)+C d_X(x_k)^2.}
\label{eq:55}
\end{equation}
\begin{enumerate}
\item \blue{If $\rho=0$, then $d_X(x_k)$ converges Q-linearly, while the
residuals and iterates converge R-linearly. In particular, the sequence
converges locally linearly for every fixed $c>0$.}
\item \blue{If $0<\rho\le1$, then}
\begin{align}
\blue{r_{k+1}}&\blue{\le C r_k^{1+\rho},\qquad
d_X(x_{k+1})\le C d_X(x_k)^{1+\rho},}
\label{eq:rn-higher-order}\\
\blue{\|x_{k+1}-x_\infty\|}
&\blue{\le C d_X(x_k)^{1+\rho}
\le C\|x_k-x_\infty\|^{1+\rho}.}
\label{eq:rn-iterate-order}
\end{align}
\blue{Thus the residual, the distance to $X$, and the iterates converge with
Q-order at least $1+\rho$.}
\end{enumerate}
\end{theorem}

\begin{proof}
\blue{Write \(d_k:=d_X(x_k)\), \(z_k:=P_X(x_k)\), and
\(y_k:=P_M(x_k)\). The local projection \(P_X\) is single-valued after
shrinking the neighborhood. Since \(x_k,z_k\in E\), metric-projection
orthogonality gives}
\[
\blue{x_k-z_k\in L\cap(T_{z_k}X^0)^\perp=W_{z_k}.}
\]
\blue{Moreover, \(r_k\le d_k\) because \(X\subset M\), and hence}
\begin{equation}
\blue{\|z_k-y_k\|\le d_k+r_k\le2d_k.}
\label{eq:rn-yz-distance}
\end{equation}
\blue{Lemma~\ref{lem:step-size} and the local error bound
\(d_k\le\kappa' r_k\) from Lemma~\ref{lem:dist-x} yield}
\[
\blue{\|x_{k+1}-z_k\|
\le\|x_{k+1}-x_k\|+d_k\le C d_k.}
\]
\blue{The orthogonal decomposition \(L=T_{z_k}X^0\oplus W_{z_k}\) therefore
gives uniquely}
\begin{equation}
\blue{x_{k+1}-z_k=\tau_k+h_k,\qquad
\tau_k\in T_{z_k}X^0,\quad h_k\in W_{z_k},\quad
\|\tau_k\|+\|h_k\|\le C d_k.}
\label{eq:rn-effective-decomposition}
\end{equation}

\blue{Set \(Q_k:=P_{N_{y_k}M^0}\). Multiplying the first-order optimality
condition by \(\mu_k\) gives}
\begin{equation}
\blue{P_L\!\left[\mu_k(x_{k+1}-y_k)+Q_k(x_{k+1}-y_k)\right]=0.}
\label{eq:rn-projected-optimality}
\end{equation}
\blue{Projecting \eqref{eq:rn-projected-optimality} onto \(W_{z_k}\), using
\(P_{W_{z_k}}\tau_k=0\), and adding and subtracting \(Q_{z_k}h_k\) yields}
\begin{equation}
\blue{(A_{z_k}+\mu_kI)h_k
=-\mu_kP_{W_{z_k}}(z_k-y_k)+R_k,}
\label{eq:rn-effective-equation}
\end{equation}
\blue{where}
\begin{equation}
\blue{\begin{aligned}
R_k={}&-P_{W_{z_k}}Q_k(z_k-y_k)
      -P_{W_{z_k}}Q_k\tau_k\\
     &-P_{W_{z_k}}(Q_k-Q_{z_k})h_k .
\end{aligned}}
\label{eq:rn-remainder}
\end{equation}
\blue{Each term is controlled directly. First, \(y_k,z_k\in M\), so
Proposition~\ref{prop:normal-second-order} and
\eqref{eq:rn-yz-distance} give}
\[
\blue{\|Q_k(z_k-y_k)\|\le C\|z_k-y_k\|^2\le Cd_k^2.}
\]
\blue{Second, \(\tau_k\in T_{z_k}X^0\subset T_{z_k}M^0=\ker Q_{z_k}\).
For a \(C^2\) embedded manifold, the map
\(y\mapsto P_{N_yM^0}\) is locally Lipschitz. Thus}
\[
\blue{\|Q_k\tau_k\|
=\|(Q_k-Q_{z_k})\tau_k\|
\le C\|y_k-z_k\|\,\|\tau_k\|\le Cd_k^2.}
\]
\blue{The same Lipschitz estimate gives}
\[
\blue{\|P_{W_{z_k}}(Q_k-Q_{z_k})h_k\|
\le Cd_k\|h_k\|.}
\]
\blue{Consequently,
\(\|R_k\|\le C(d_k^2+d_k\|h_k\|)\). By
Lemma~\ref{lem:effective-coercivity},}
\[
\blue{\|(A_{z_k}+\mu_kI)^{-1}\|_{W_{z_k}\to W_{z_k}}
\le\lambda_*^{-1}.}
\]
\blue{Equation~\eqref{eq:rn-effective-equation} now yields an expansion
that retains its exact linear part. Indeed,
$DP_M(z_k)=P_{T_{z_k}M^0}$, and the local $C^2$ manifold geometry gives the
following first-order expansion with a uniform second-order remainder:}
\begin{equation}
\blue{
y_k-z_k
=P_{T_{z_k}M^0}(x_k-z_k)+O(d_k^2)
=(I-Q_{z_k})e_k+O(d_k^2),
}
\label{eq:metric-projection-expansion}
\end{equation}
\blue{where $e_k=x_k-z_k\in W_{z_k}$. Moreover, self-adjointness of
$A_{z_k}$ together with $\ker A_{z_k}=T_{z_k}X^0$ implies
$A_{z_k}W_{z_k}\subset W_{z_k}$, and hence}
\[
\blue{
P_{W_{z_k}}(I-Q_{z_k})e_k=(I-A_{z_k})e_k.
}
\]
\blue{Substituting into \eqref{eq:rn-effective-equation} yields}
\[
\blue{
(A_{z_k}+\mu_kI)h_k
=\mu_k(I-A_{z_k})e_k+O(\mu_kd_k^2)+R_k.
}
\]
\blue{For $r_k\le1$ and $0\le\rho\le1$, we have $\mu_k\le c$, including
$\mu_k=c$ when $\rho=0$. Using the uniform inverse bound above,
$\|e_k\|=d_k$, and
$\|R_k\|\le C(d_k^2+d_k\|h_k\|)$, we obtain, after shrinking the
neighborhood if necessary,}
\[
\blue{
\|h_k\|
\le C\bigl(d_k+d_k^2+d_k\|h_k\|\bigr)
\le C\bigl(d_k+d_k\|h_k\|\bigr).
}
\]
\blue{Absorbing the last term gives $\|h_k\|=O(d_k)$, and hence
$\|R_k\|=O(d_k^2)$. Since also $O(\mu_kd_k^2)=O(d_k^2)$, applying
$(A_{z_k}+\mu_kI)^{-1}$ and recalling the definition of
$S_{z_k}(\mu_k)$ gives}
\begin{equation}
\blue{
h_k=S_{z_k}(\mu_k)e_k+O(d_k^2).
}
\label{eq:rn-effective-expansion}
\end{equation}

\blue{Uniformly for \(z\in X\) near \(\bar x\), the submanifold
\(X\subset E\) has the local graph representation}
\[
\blue{u\longmapsto z+u+\varphi_z(u),\qquad
u\in T_zX^0,\qquad
\varphi_z(u)\in W_z,\qquad
\|\varphi_z(u)\|\le C\|u\|^2.}
\]
\blue{Taking \(z=z_k\) and \(u=\tau_k\), the resulting point belongs to
\(X\), so \eqref{eq:rn-effective-decomposition},
\eqref{eq:rn-effective-expansion}, and
\eqref{eq:effective-error-contraction} give}
\[
\blue{d_X(x_{k+1})
\le\|h_k-\varphi_{z_k}(\tau_k)\|
\le\vartheta(\mu_k)d_k+Cd_k^2,}
\]
\blue{which proves \eqref{eq:55}. If $\rho=0$, then $\mu_k=c$ and
$\vartheta(c)<1$ for every fixed $c>0$. Choose
$\theta\in(\vartheta(c),1)$ and shrink the neighborhood so that
$Cd_k\le\theta-\vartheta(c)$. Then $d_{k+1}\le\theta d_k$. If
$0<\rho\le1$, then}
\[
\blue{\vartheta(\mu_k)\le\frac{1-\lambda_*}{\lambda_*}\mu_k
=O(r_k^\rho)=O(d_k^\rho).}
\]
\blue{Since $r_k\le d_k\le\kappa' r_k$ and
$d_k^2\le d_k^{1+\rho}$ locally, \eqref{eq:55} gives}
\[
\blue{d_{k+1}\le Cd_k^{1+\rho},\qquad
r_{k+1}\le d_{k+1}\le Cr_k^{1+\rho},}
\]
\blue{which is \eqref{eq:rn-higher-order}.}

\blue{It remains to close the local argument and establish the limit. In either
parameter regime, choose a ball
\(U:=B(\bar x,\delta)\) sufficiently small that all preceding estimates hold
on \(U\) and there is a uniform \(\theta\in(0,1)\) such that
\(d_{k+1}\le\theta d_k\). Since
\(d_0\le\|x_0-\bar x\|\), choose a smaller neighborhood
\(U_0\subset U\) of \(\bar x\) such that every \(x_0\in E\cap U_0\) satisfies}
\begin{equation}
\blue{\|x_0-\bar x\|+\frac{C_1}{1-\theta}d_0<\delta,}
\label{eq:rn-invariance-radius}
\end{equation}
\blue{where \(C_1\) is the constant in Lemma~\ref{lem:step-size}. Induction
using \eqref{eq:rn-invariance-radius} gives}
\[
\blue{d_k\le\theta^k d_0,\qquad
\|x_k-\bar x\|
\le\|x_0-\bar x\|+C_1\sum_{j=0}^{k-1}d_j<\delta.}
\]
\blue{Thus every iterate remains in the local neighborhood. The step-size
estimate makes \(\sum_k\|x_{k+1}-x_k\|\) finite, so \(x_k\) converges to
some \(x_\infty\in E\). Since \(r_k\le d_k\to0\) and both \(E\) and
\(M\) are locally closed, \(x_\infty\in X\). For $\rho=0$, the bounds
$r_k\le d_k\le\theta^k d_0$ and}
\[
\blue{\|x_k-x_\infty\|\le C\sum_{j=k}^{\infty}d_j
\le C\theta^k d_0}
\]
\blue{prove R-linear convergence of the residuals and iterates, while
$d_{k+1}\le\theta d_k$ is Q-linear. For $0<\rho\le1$, the geometric tail and
the higher-order distance recurrence yield}
\[
\blue{\begin{aligned}
\|x_{k+1}-x_\infty\|
&\le\sum_{j=k+1}^{\infty}\|x_{j+1}-x_j\|
 \le C_1\sum_{j=k+1}^{\infty}d_j\\
&\le\frac{C_1}{1-\theta}d_{k+1}
 \le Cd_k^{1+\rho}
 \le C\|x_k-x_\infty\|^{1+\rho},
\end{aligned}}
\]
\blue{where the last inequality uses
\(d_k=d_X(x_k)\le\|x_k-x_\infty\|\). If a zero residual occurs, the
algorithm terminates at a point of \(X\), and the conclusions hold after
extending the sequence constantly.}
\end{proof}

\begin{corollary}[Quadratic convergence for \(\rho=1\)]
\label{cor:quadratic-trans}
\blue{Under the assumptions of Theorem~\ref{thm:superlinear-trans}, if
\(\mu_k=cr_k\), then the residual, the distance to \(X\), and the iterates
converge locally Q-quadratically.}
\end{corollary}

\begin{proof}
\blue{This is the case \(\rho=1\) of
Theorem~\ref{thm:superlinear-trans}.}
\end{proof}

\begin{theorem}[\blue{Deviation of the RN-SLRA limit from the metric projection}]
\label{thm:lim-point-rn}
\blue{Under the assumptions of Theorem~\ref{thm:superlinear-trans}, there
exists \(\gamma>0\) such that, for every sufficiently close
\(x_0\in E\),}
\begin{equation}
\|x_\infty(x_0)-P_X(x_0)\|
\le\gamma\|x_0-P_X(x_0)\|.
\label{eq:rn_x_inf}
\end{equation}
\end{theorem}

\begin{proof}
\blue{Let \(p:=P_X(x_0)\). The proof of
Theorem~\ref{thm:superlinear-trans} provides a uniform
\(\theta\in(0,1)\) such that \(d_k\le\theta^kd_0\), while
Lemma~\ref{lem:step-size} gives
\(\|x_{k+1}-x_k\|\le C_1r_k\le C_1d_k\). Since
\(d_0=\|x_0-p\|\), it follows that}
\[
\blue{\begin{aligned}
\|x_\infty-p\|
&\le\|x_0-p\|+\sum_{k=0}^{\infty}\|x_{k+1}-x_k\|\\
&\le\|x_0-p\|+C_1\sum_{k=0}^{\infty}d_k
\le\left(1+\frac{C_1}{1-\theta}\right)\|x_0-p\|.
\end{aligned}}
\]
\blue{This proves \eqref{eq:rn_x_inf}.}
\end{proof}

\begin{remark}[First-order retraction-like interpretation under clean intersection]
\label{rmk:retraction-clean-intersection}
\blue{Following the retraction viewpoint for locally convergent alternating-projection-type methods developed in \cite{chen2026retractionsalternatingprojections}, we interpret the RN-SLRA limit map as follows.}
For \(x\in X\) near \(\bar x\) and
\(\blue{\xi\in T_xX^0}\) sufficiently small, define
\(R_x(\xi):=x_\infty(x+\xi)\). Since \(X\) is a \(C^2\) embedded
submanifold,
\[
P_X(x+\xi)=x+\xi+O(\|\xi\|^2),\qquad
d_X(x+\xi)=O(\|\xi\|^2).
\]
\blue{Theorem~\ref{thm:lim-point-rn} then gives
\(\|R_x(\xi)-P_X(x+\xi)\|=O(\|\xi\|^2)\), and hence
\(R_x(\xi)=x+\xi+O(\|\xi\|^2)\). This is a first-order
retraction-like expansion.} \blue{For fixed regularization ($\rho=0$),
Appendix~\ref{app:rho-zero-retraction} further shows that the limit map defines
a genuine retraction and, under additional smoothness, a second-order
retraction.}
\end{remark}
\color{black}
\subsection{\texorpdfstring{\blue{Local Convergence of iRN-SLRA under Intrinsic Transversality}}{Local Convergence of iRN-SLRA under Intrinsic Transversality}}

\leavevmode\blue{We next show that the effective-error-space argument extends to
quasioptimal manifold projections. For every fixed $\sigma\ge1$ and $0<\rho\le1$, quasioptimality changes
only the local constants and does not alter the order $1+\rho$. The
fixed-regularization case requires the additional contraction condition
$q\sigma<1$ and is treated separately in Proposition~\ref{prop:irn-fixed}.}

Throughout this subsection, let
\[
\widetilde r_k:=\|x_k-y_k\|,\qquad
r_k:=d_M(x_k),\qquad
\mu_k:=c\widetilde r_k^\rho,\quad c>0,\quad0\le\rho\le1.
\]

\begin{lemma}[\blue{Equivalence of exact and quasioptimal residuals}]
\label{lem:inexact-residual-equivalence}
If \(y_k\in M\) satisfies
\(\|x_k-y_k\|\le\sigma d_M(x_k)\) for a fixed \(\sigma\ge1\), then
\[
r_k\le\widetilde r_k\le\sigma r_k.
\]
\blue{If $\rho=0$, then $\mu_k=c$;} if $0<\rho\le1$, then
\[
cr_k^\rho\le\mu_k\le c\sigma^\rho r_k^\rho.
\]
\end{lemma}

\begin{proof}
\blue{
The first lower bound follows from the definition of \(d_M\), and the
upper bound is the quasioptimality condition. The bounds for \(\mu_k\) follow
by raising the residual bounds to the positive power $\rho$\blue{; for $\rho=0$,
the convention $\widetilde r_k^0=1$ gives $\mu_k=c$}.
}
\end{proof}

\begin{lemma}[Step-size estimate for iRN-SLRA]
\label{lem:inexact-step-size}
Suppose that Assumptions~\ref{ass:smooth} and~\ref{ass:intrinsic} hold.
\blue{There exist \(\delta>0\) and \(C_{\rm i}>0\) such that every 
iRN-SLRA iterate \(x_k\in E\cap B(\bar x,\delta)\) satisfies}
\[
\blue{\|x_{k+1}-y_k\|\le C_{\rm i}\widetilde r_k,\qquad
\|x_{k+1}-x_k\|\le(C_{\rm i}+1)\widetilde r_k.}
\]
\end{lemma}

\begin{proof}
\blue{
Since \(x_k\in E\), \(d_E(y_k)\le\widetilde r_k\).
Lemma~\ref{lem:dist-y} therefore gives a point \(z_k\in P_X(y_k)\) with
\(\|z_k-y_k\|\le\kappa\widetilde r_k\). Using \(z_k\) as a comparison
point in the iRN-SLRA objective and applying
Proposition~\ref{prop:normal-second-order}, we obtain
\[
\Phi_k(z_k)
\le C\widetilde r_k^2+C\mu_k^{-1}\widetilde r_k^4
=C\widetilde r_k^2+C\widetilde r_k^{4-\rho}
\le C'\widetilde r_k^2,
\]
where \(0\le\rho\le1\) and the neighborhood has been shrunk so that
\(\widetilde r_k\le1\). Optimality gives
\(\frac12\|x_{k+1}-y_k\|^2\le\Phi_k(z_k)\), proving the first bound; the
second follows from the triangle inequality.
}
\end{proof}

\begin{theorem}[\blue{Local higher-order convergence of iRN-SLRA under intrinsic transversality}]
\label{thm:irn-higher-order}
\blue{Suppose that Assumptions~\ref{ass:smooth} and~\ref{ass:intrinsic} hold.
Let \(y_k\in M\) satisfy
\(\|x_k-y_k\|\le\sigma d_M(x_k)\) for a fixed \(\sigma\ge1\), and choose
\(\mu_k=c\widetilde r_k^\rho\) with \(c>0\) and \(0<\rho\le1\). Then
there exist neighborhoods \(U_0\subset U\) of \(\bar x\) and \(C>0\)
such that every sequence starting from \(x_0\in E\cap U_0\) remains in
\(U\) and converges to a point \(x_\infty^{\rm i}\in X\).
For every nonterminal iteration,}
\begin{align}
\blue{d_X(x_{k+1})}
&\blue{\le C\bigl(\mu_kd_X(x_k)+d_X(x_k)^2\bigr),}
\label{eq:irn-clean-structural}\\
\blue{r_{k+1}}&\blue{\le Cr_k^{1+\rho},\qquad
d_X(x_{k+1})\le Cd_X(x_k)^{1+\rho},}
\label{eq:irn-higher-order}\\
\blue{\|x_{k+1}-x_\infty^{\rm i}\|}
&\blue{\le Cr_k^{1+\rho}
\le C\|x_k-x_\infty^{\rm i}\|^{1+\rho}.}
\label{eq:irn-iterate-order}
\end{align}
\blue{Hence the residual, the distance to \(X\), and the iterates converge
locally with Q-order at least \(1+\rho\), and Q-quadratically when
\(\rho=1\).}
\end{theorem}

\begin{proof}
\blue{Set
\(
d_k:=d_X(x_k), z_k:=P_X(x_k)\)
and \(
Q_k:=P_{N_{y_k}M^0}.
\)
Since \(x_k,z_k\in E\), metric-projection orthogonality gives
\(x_k-z_k\in W_{z_k}\). Moreover,
Lemma~\ref{lem:inexact-residual-equivalence} and \(r_k\le d_k\) imply
\begin{equation}
\|y_k-z_k\|
\le \widetilde r_k+d_k
\le (1+\sigma)d_k.
\label{eq:irn-yz}
\end{equation}
Together with Lemma~\ref{lem:inexact-step-size}, this gives
\[
x_{k+1}-z_k=\tau_k+h_k,\qquad
\tau_k\in T_{z_k}X^0,\quad h_k\in W_{z_k},
\qquad
\|\tau_k\|+\|h_k\|\le Cd_k.
\]

The effective-error-space calculation in the proof of
Theorem~\ref{thm:superlinear-trans} uses only
\(y_k,z_k\in M\), \(\|y_k-z_k\|=O(d_k)\),
\(\|x_{k+1}-z_k\|=O(d_k)\), and the local Lipschitz continuity of
\(y\mapsto P_{N_yM^0}\). These properties remain valid here, so repeating the same decomposition gives
\begin{equation}
(A_{z_k}+\mu_kI)h_k
=-\mu_kP_{W_{z_k}}(z_k-y_k)+R_k^{\rm i},
\label{eq:irn-effective-equation}
\end{equation}
where
\[
\|R_k^{\rm i}\|
\le C\bigl(d_k^2+d_k\|h_k\|\bigr).
\]
By Lemma~\ref{lem:effective-coercivity} and
\eqref{eq:irn-yz},
\[
\|h_k\|
\le C\bigl(\mu_kd_k+d_k^2+d_k\|h_k\|\bigr).
\]
After shrinking the neighborhood and absorbing the last term, we obtain
\begin{equation}
\|h_k\|\le C\bigl(\mu_kd_k+d_k^2\bigr).
\label{eq:irn-effective-bound}
\end{equation}

Unlike the exact-projection case, the first-order expansion
\eqref{eq:metric-projection-expansion} is not available for a quasioptimal
point \(y_k\). For \(0<\rho\le1\), however,
\eqref{eq:irn-effective-bound} is sufficient. Indeed, the local graph
representation of \(X\subset E\), used as in the proof of
Theorem~\ref{thm:superlinear-trans}, gives
\[
d_X(x_{k+1})
\le C\bigl(\mu_kd_k+d_k^2\bigr).
\]
Since
\[
r_k\le d_k\le\kappa' r_k,
\qquad
\mu_k\le c\sigma^\rho r_k^\rho,
\]
we obtain
\[
d_X(x_{k+1})\le Cd_k^{1+\rho},
\qquad
r_{k+1}\le Cr_k^{1+\rho}.
\]

The local invariance, summability, and geometric-tail arguments at the end
of the proof of Theorem~\ref{thm:superlinear-trans} apply verbatim: after
fixing a sufficiently small neighborhood \(U\) on which the preceding
estimates hold, one chooses a smaller initialization neighborhood
\(U_0\subset U\). Here the RN-SLRA step-size bound is replaced by
\[
\|x_{k+1}-x_k\|
\le (C_{\rm i}+1)\sigma r_k.
\]
Thus every sequence starting from \(x_0\in E\cap U_0\) remains in \(U\),
and consequently \(x_k\to x_\infty^{\rm i}\in X\). Moreover,
\[
\begin{aligned}
\|x_{k+1}-x_\infty^{\rm i}\|
&\le C\sum_{j=k+1}^{\infty}r_j
 \le Cr_{k+1}                                      \\
&\le Cr_k^{1+\rho}
 \le C\|x_k-x_\infty^{\rm i}\|^{1+\rho}.
\end{aligned}
\]
A zero residual corresponds to finite termination at a point of \(X\).}
\end{proof}

\blue{The fixed-regularization regime uses a different mechanism: a local
one-step alternating-projection decrease controls the leading term, while the
regularized correction contributes only at second order.}

\begin{proposition}[\blue{Fixed-regularization case for iRN-SLRA}]
\label{prop:irn-fixed}
\blue{Suppose that Assumptions~\ref{ass:smooth} and~\ref{ass:intrinsic} hold,
and let $\rho=0$, so that $\mu_k\equiv c>0$. Let $q\in(0,1)$ denote a
uniform local one-step alternating-projection decrease constant in the order
$M\to E\to M$. If $q\sigma<1$, then every iRN-SLRA sequence starting
sufficiently close to $\bar x$ remains local and converges linearly to a point
$x_\infty^{\rm i}\in X$.}
\end{proposition}

\begin{proof}
\blue{
By Lemma~\ref{lem:equiv-sep-intrinsic}, intrinsic transversality is
equivalent to separability in the present manifold--affine setting.
Therefore, the local one-step alternating-projection decrease result
\cite{Drusvyatskiy2019} applies. Thus there is $q\in(0,1)$ such that, with
$p_k:=P_E(y_k)$,
\[
d_M(p_k)\le qd_E(y_k)
\le q\|y_k-x_k\|\le q\sigma r_k.
\]
Set $Q_k:=P_{N_{y_k}M^0}$, $b_k:=p_k-y_k$, and
$s_k:=x_{k+1}-p_k\in L$. Since $b_k\perp L$, comparison of the iRN-SLRA
objective at $x_{k+1}$ and $p_k$ gives
\[
\frac12\|s_k\|^2+\frac{1}{2c}\|Q_k(x_{k+1}-y_k)\|^2
\le \frac{1}{2c}\|Q_kb_k\|^2.
\]
Consequently, using Proposition~\ref{prop:normal-second-order}, the iRN-SLRA
step-size estimate, and $\|b_k\|=d_E(y_k)\le\widetilde r_k\le\sigma r_k$,
\[
\begin{aligned}
r_{k+1}
&\le \|Q_k(x_{k+1}-y_k)\|+G\|x_{k+1}-y_k\|^2\\
&\le \|Q_kb_k\|+Cr_k^2
 \le d_M(p_k)+Cr_k^2
 \le q\sigma r_k+Cr_k^2.
\end{aligned}
\]
Choose $\theta\in(q\sigma,1)$ and shrink the neighborhood so that
$r_{k+1}\le\theta r_k$. Together with
$\|x_{k+1}-x_k\|\le(C_{\rm i}+1)\sigma r_k$, the same local induction and
geometric-tail argument as Theorem~\ref{thm:superlinear-trans} shows that the iterates remain in the
neighborhood, their steps are summable, and
$x_k\to x_\infty^{\rm i}\in X$ linearly. 
}
\end{proof}

\begin{corollary}[\blue{Limit-point accuracy of iRN-SLRA}]
\label{thm:inexact-limit-map}
\blue{Under the assumptions of either Theorem~\ref{thm:irn-higher-order} or
Proposition~\ref{prop:irn-fixed}, there
exists \(\gamma^{\rm i}>0\) such that, for every sufficiently close initial point $x_0\in E$ and
every admissible iRN-SLRA sequence starting from $x_0$ with limit
$x_\infty^{\rm i}\in X$, }
\begin{equation}
\|x_\infty^{\rm i}-P_X(x_0)\|
\le\gamma^{\rm i}\|x_0-P_X(x_0)\|.
\label{eq:irn_x_inf}
\end{equation}
\end{corollary}

\begin{proof}
\blue{Let \(p:=P_X(x_0)\). The proofs of
Theorem~\ref{thm:irn-higher-order} and Proposition~\ref{prop:irn-fixed} each
give \(r_k\le\theta^kr_0\) for a uniform \(\theta\in(0,1)\) under their
respective assumptions. Lemmas~\ref{lem:inexact-residual-equivalence}
and~\ref{lem:inexact-step-size} give
\(\|x_{k+1}-x_k\|\le C\sigma r_k\), while
\(r_0\le\|x_0-p\|\). Summing the geometric tail proves
\eqref{eq:irn_x_inf}.}
\end{proof}
\blue{
\begin{remark}[First-order retraction-like interpretation for iRN-SLRA
under a fixed quasioptimal selection]
\label{rmk:inexact-retraction}
Fix a single-valued local $\sigma$-quasioptimal projection selection
\[
    \widetilde P_M:E\cap U\to M,
    \qquad
    \|z-\widetilde P_M(z)\|
    \le \sigma d_M(z),
\]
and consider iRN-SLRA with
\[
    y_k=\widetilde P_M(x_k)
\]
at every iteration. For every sufficiently close initial point
$z\in E$, the resulting  sequence converges to a uniquely
 limit, which we denote by
\[
    \Psi^{\rm i}_{\widetilde P_M}(z)\in X.
\]
By Corollary~\ref{thm:inexact-limit-map},
\[
    \|\Psi^{\rm i}_{\widetilde P_M}(z)-P_X(z)\|
    \le
    \gamma^{\rm i}\|z-P_X(z)\|.
\]

Now let $x\in X$ be near $\bar x$ and
$\xi\in T_xX^0$ be sufficiently small. Since
\[
    P_X(x+\xi)=x+\xi+O(\|\xi\|^2),
    \qquad
    d_X(x+\xi)=O(\|\xi\|^2),
\]
we obtain
\[
    \|\Psi^{\rm i}_{\widetilde P_M}(x+\xi)
      -P_X(x+\xi)\|
    =O(\|\xi\|^2),
\]
and hence
\[
    \Psi^{\rm i}_{\widetilde P_M}(x+\xi)
    =x+\xi+O(\|\xi\|^2).
\]
Thus, for every fixed single-valued $\sigma$-quasioptimal projection
selection $\widetilde P_M$, the corresponding iRN-SLRA limit map
satisfies the same first-order retraction-like expansion as in
Remark~\ref{rmk:retraction-clean-intersection}.
\end{remark}
}

\section{Experimental Results}
In this section, we test our algorithm on some applications of structured low-rank approximation, and compare it with representative existing methods. \footnote{Our Python code is available at \url{https://github.com/reniusll/Regularized-Newton-SLRA}}
\blue{We first study an analytically verified clean but nontransversal problem,
then consider larger synthetic and Hankel SLRA instances. Exact geometry for
the small-scale problem and sampled intrinsic transversality diagnostics for the larger
experiments are reported in Appendix~A.}
All experiments were conducted in Python 3.9.13 on an Ubuntu 22.04 server equipped with dual AMD EPYC 9754 128-core processors (512 logical threads) and 1 TB of RAM.

\subsection{\texorpdfstring{\blue{A Clean but Nontransversal Small-Scale SLRA Problem}}{A Clean but Nontransversal Small-Scale SLRA Problem}}

\blue{We construct an artificial affine subspace intersecting the $4\times4$
rank-$2$ manifold $\mathcal D_2$ cleanly but nontransversally. The construction
directly tests the higher-order theory while making the classical
tangent-intersection system singular. We set
$E=x^*+\operatorname{span}(B_1,\ldots,B_4)$, where the $B_i$ are obtained by
orthonormalizing the raw directions below.}

\blue{We take $x^*=\operatorname{diag}(4,0.02,0,0)$ and define the raw
affine directions
\[
\begin{aligned}
\widetilde B_1&=E_{11},&
\widetilde B_2&=E_{13}+0.7E_{31}+E_{33},\\
\widetilde B_3&=E_{24}-0.6E_{42}+E_{44},&
\widetilde B_4&=E_{34}+0.2E_{12}.
\end{aligned}
\]
The actual basis $B_1,\ldots,B_4$ is obtained by Gram--Schmidt
orthonormalization. We use
\[
x_0=x^*+0.02\sum_{i=1}^4a_iB_i,\qquad
a=(-0.98912135,-0.36778665,1.28792526,0.19397442).
\]
Appendix~A.2 proves analytically that there are a neighborhood $U$ of $x^*$
and $\varepsilon>0$ such that
\[
E\cap\mathcal D_2\cap U
=\{x^*+tE_{11}:|t|<\varepsilon\},
\]
and that this intersection is clean but not transversal. We set
$c=3\times10^{-7}$ and $\rho=1$.}

\begin{figure}[htbp]
  \centering
  \includegraphics[width=0.9\textwidth]{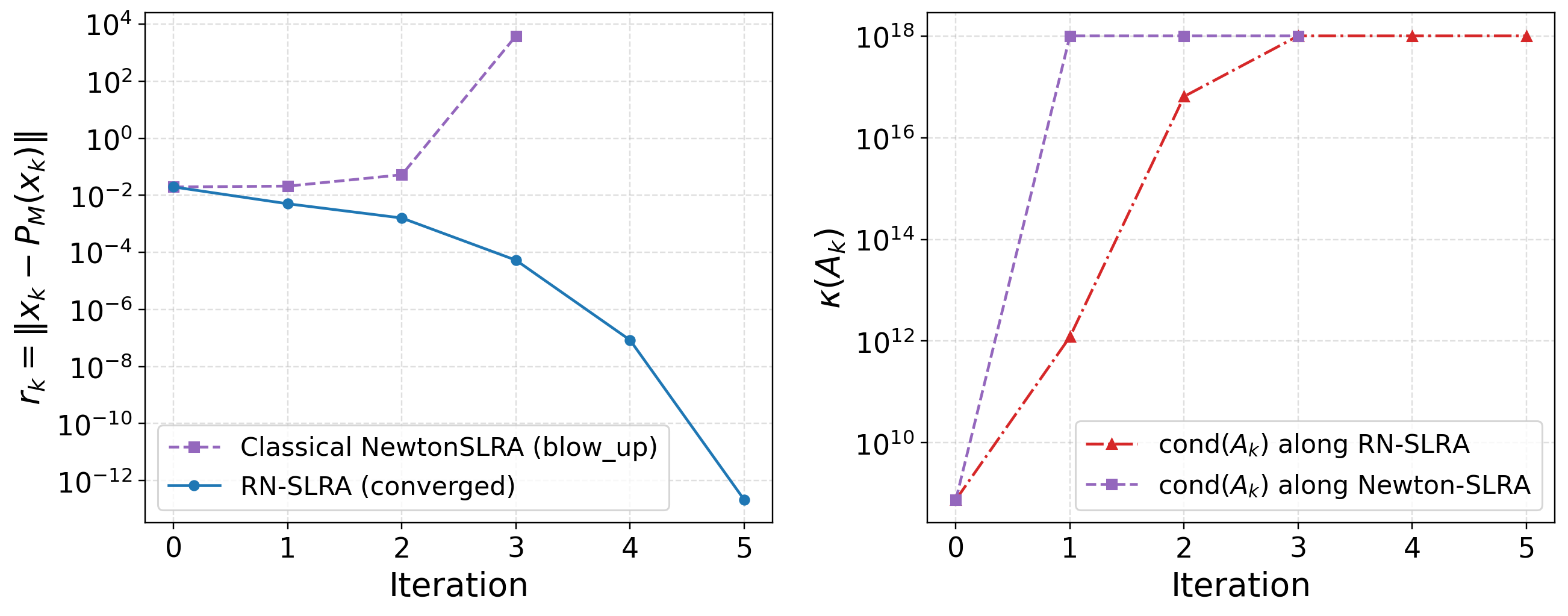}
  \caption{\blue{A clean but nontransversal small-scale SLRA problem. Left: rank residuals of
  RN-SLRA and Newton-SLRA. Right: condition numbers of the classical
  tangent-intersection systems; Condition numbers are clipped at \(10^{18}\) for visualization. 
  RN-SLRA exhibits the predicted quadratic behavior, while the Newton system
  becomes numerically singular in this run.}}
  \label{fig:Newton-SLRA-vs-RN-SLRA}
\end{figure}

\blue{The results are shown in Fig.~\ref{fig:Newton-SLRA-vs-RN-SLRA}.
RN-SLRA reduces the rank residual from $1.97\times10^{-2}$ to
$2.14\times10^{-13}$ after five updates. In the final two local steps,
\[
\frac{r_4}{r_3^2}=29.6042,\qquad
\frac{r_5}{r_4^2}=29.9994.
\]
Since the local solution branch is known explicitly, we also compute the
Frobenius distance from $x_k$ to its representing affine line. Its final two
quadratic ratios are $19.3415$ and $19.5281$. The finite, nearly stable values
of both sets of ratios are consistent with the quadratic estimates underlying
Theorem~\ref{thm:superlinear-trans} for $\rho=1$.}

\blue{The same example exposes the loss of numerical stability of the
classical Newton step. The singular values of $A(x^*)$ are
\(
(0.98058,0.65094,0.63372,0),
\)
so the tangent-intersection system is singular at the solution. Along the
Newton-SLRA trajectory, its condition number grows from $7.63\times10^8$ to
$7.79\times10^{20}$ and $2.61\times10^{33}$ before becoming numerically
singular. Newton-SLRA diverges in this run, whereas RN-SLRA remains stable and
its residual sequence displays the predicted quadratic behavior.}

\subsection{Large-scale SLRA Problems}
We construct the intersection point \(x^*\) and the affine subspace \(E\) in large-scale experiments using a similar method as before, and choose the initial point as \(x_0=P_E(x^*+\alpha G)\), where \(G_{ij}\sim N(0,1)\) are i.i.d. We compare
RN-SLRA and Cadzow in nine large-scale experiments; see Fig.~\ref{fig:large-scale}, where \(n,m\) are the matrix dimensions, \(p=\dim E\), and \(r\) is the target
rank. For each tuple \((n,m,p,r)\), both methods use the same initial point,
with \(\alpha=0.05\) for the \(100\times100\) cases, \(\alpha=0.02\) for the
larger cases, and stopping tolerance \(10^{-8}\).

\begin{figure}[htbp]
  \centering
  \includegraphics[width=0.9\textwidth]{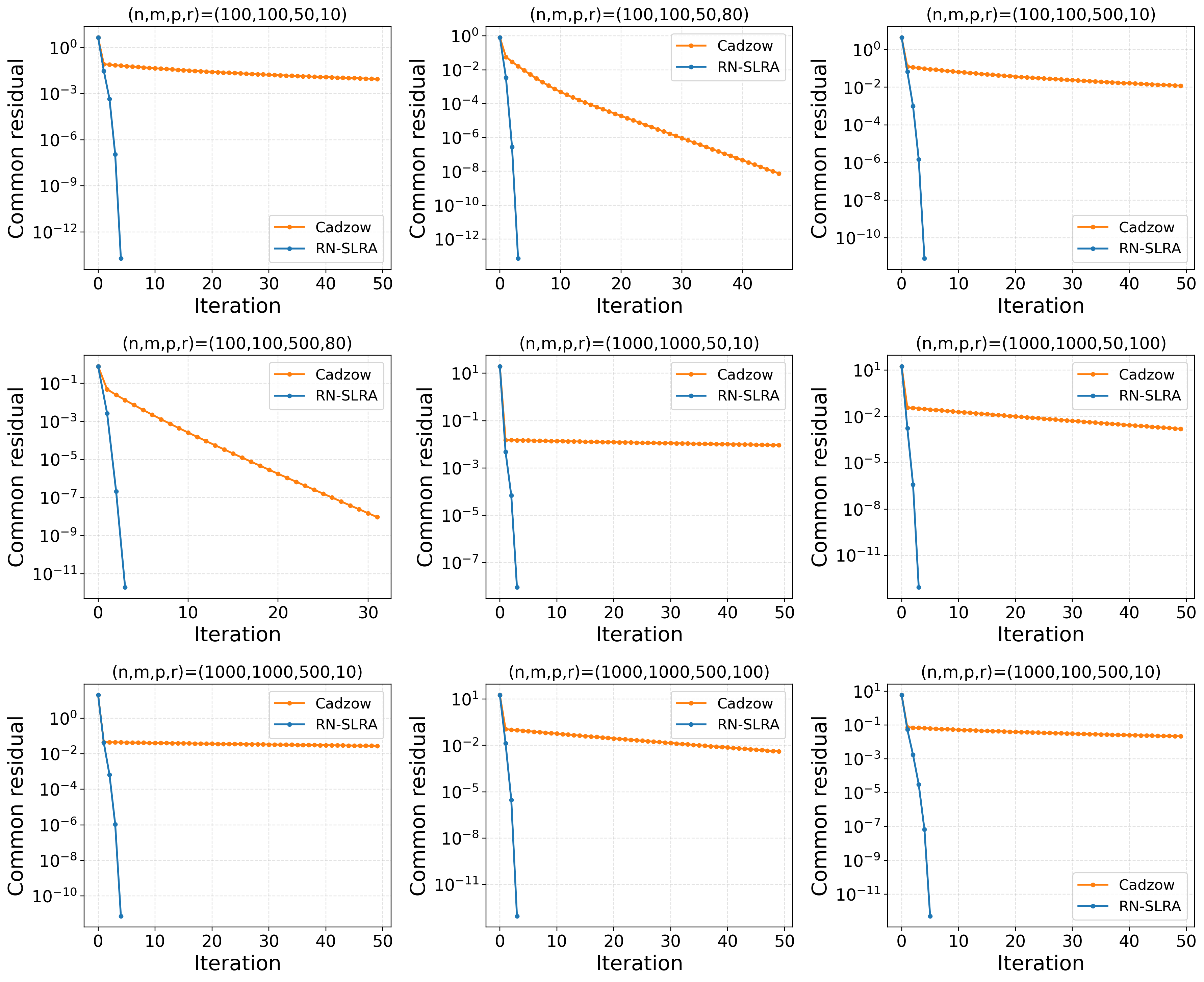}
  \caption{Large-scale SLRA Problems.}
  \label{fig:large-scale}
\end{figure}

As shown in Fig.~\ref{fig:large-scale}, RN-SLRA converges within five iterations in
all cases, whereas Cadzow converges much more slowly.

\subsection{Low-Rank Approximation of Hankel Matrices}
We now test the performance of RN-SLRA for low-rank approximation of Hankel
matrices \cite{ParkZhangRosen1999}. We compare RN-SLRA with four representative
baseline families for Hankel structured low-rank approximation: Cadzow; VP-SLRA,
based on variable projection and reduced nonlinear optimization for affinely
structured SLRA \cite{MarkovskyUsevich2014SoftwareSLRA}; STLN, which approaches
the problem from a structured error-correction perspective
\cite{RosenParkGlick1996STLN,RosenParkGlick1998}; and Gradient-system, based on
low-rank factor updates with a Hankel penalty
\cite{FazziGuglielmiMarkovsky2021GradientSystem}.


We will construct a single example and then perform both a noise sweep and an outlier sweep to compare RN-SLRA with the baselines on \blue{spectrally ill-conditioned} Hankel low-rank approximation problems.

\subsubsection{Single example}\leavevmode\par
We consider a \(7\times5\) rank-\(4\) Hankel matrix
\(H_c=(\nu_{i+j-1})\), where
\(\nu_i=\sum_{\ell=1}^4\beta_\ell z_\ell^i\) for \(i=1,\ldots,11\).
Here \(\beta=(1.2,0.9,1.0,1.1)\) and
\(z=(e^{-0.2},e^{-0.2+10^{-3}},e^{-0.35},e^{-0.5})\).
\blue{The closeness of the first two poles makes the instance spectrally ill-conditioned
and very close to a lower-rank stratum, although the intersection remains transversal.}

The perturbed initial matrix is generated as 
  \begin{align}
    \label{eq:H0}
    H_0=H_c+\tau \Delta + \eta E_{\text{out}},
  \end{align}
where $\Delta$ is a random Hankel noise matrix, $E_{\text{out}}$ is a Hankel matrix supported on the 9-th anti-diagonal, and in this example we choose $\tau=3\times 10^{-3},\eta=1\times10^{-2}$. $\Delta$ is a Hankel matrix with generating entries picked independently from $\mathcal N(0,1)$.

We set the parameter of RN-SLRA to $c=1\times 10^{-7},\rho=1$. We terminate an algorithm if it stalls or if its iteration count reaches the
maximum number 300. The numerical results are presented in Fig.~\ref{fig:Hankel-example}.
\begin{figure}[htbp]
  \centering
  \includegraphics[width=0.4\textwidth]{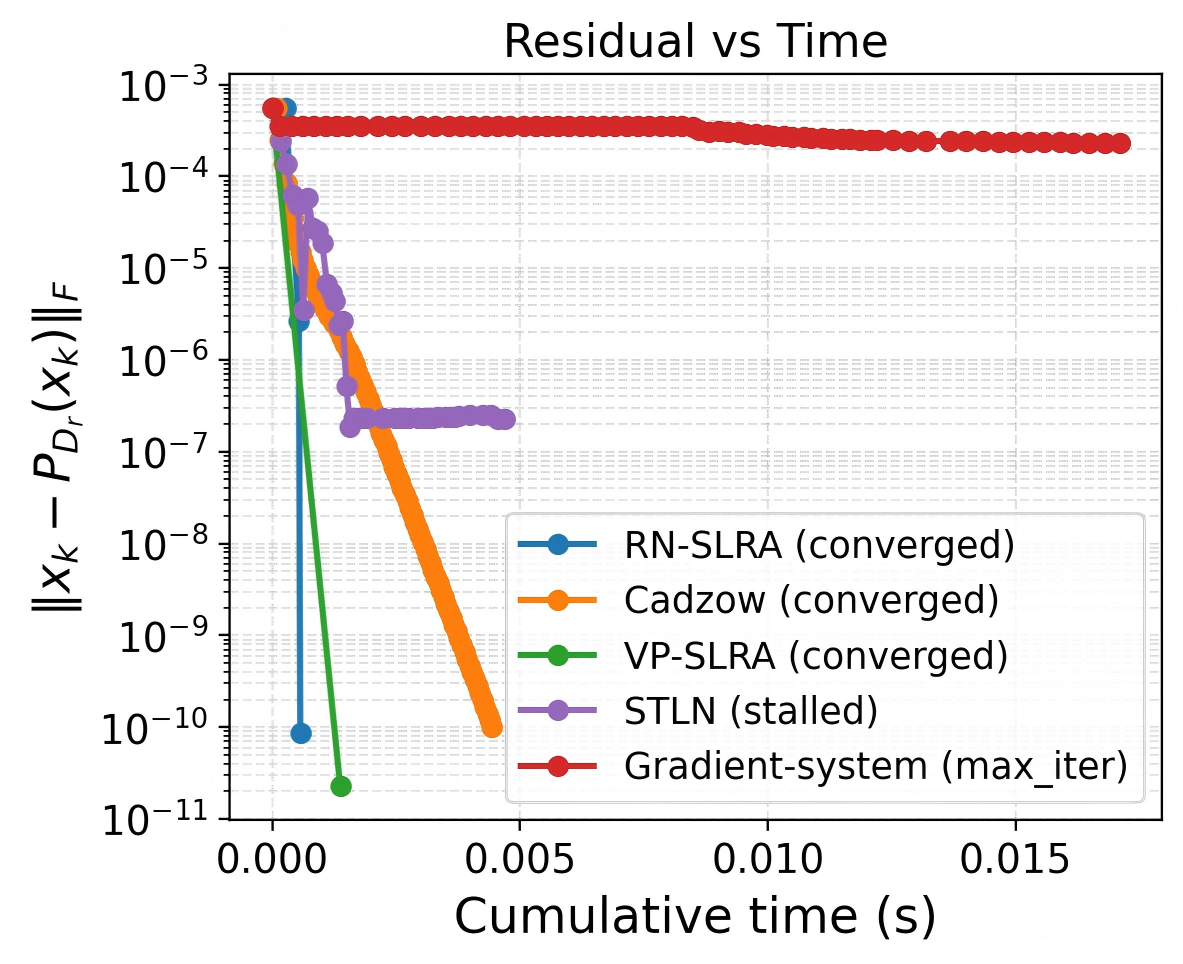}
  \caption{Low-Rank Approximation of Hankel Matrices.}
  \label{fig:Hankel-example}
\end{figure}

This result shows that RN-SLRA still converges within only a few iterations and rapidly reduces the residual to high accuracy. VP-SLRA also converges to a very small residual, but requires more CPU time than RN-SLRA. Cadzow converges more slowly, while STLN stalls at a moderate residual level, and the gradient-system method fails to make substantial progress within the prescribed maximum number of iterations. These results demonstrate that RN-SLRA remains effective for \blue{spectrally
ill-conditioned} structured low-rank approximation problems, achieving fast convergence and competitive accuracy.


\subsubsection{Noise Sweep and Outlier Sweep}

We next test the robustness of RN-SLRA under different noise and outlier
levels. In both sweeps, the initial point is generated by \eqref{eq:H0}, and
several independent random Hankel noises \(\Delta\) are used. For the noise
sweep, the outlier level \(\eta\) is fixed and
\(\tau\in\{10^{-4},3\times10^{-4},10^{-3},3\times10^{-3}\}\). For the outlier
sweep, the noise level \(\tau\) is fixed and
\(\eta\in\{0,10^{-3},3\times10^{-3},10^{-2}\}\). The mean CPU times over 30
random perturbations are reported in Table~\ref{tab:Hankel-sweeps}.  
Here “fail” means that the algorithm does not reach the prescribed tolerance
within the iteration limit.

\begin{table}[htbp]
\centering
\caption{Mean CPU times with sample standard deviations in seconds for the Hankel noise and outlier sweeps over
30 random perturbations. Here RN, VP, and GS denote RN-SLRA, VP-SLRA, and
Gradient-system, respectively.}
\label{tab:Hankel-sweeps}
\setlength{\tabcolsep}{4pt}
\footnotesize
\begin{tabular}{cc|ccccc}
\hline
~Sweep~ & ~Level~ & ~RN~ & ~Cadzow~ & ~VP~ & ~STLN~ & ~GS~ \\
\hline
Noise & $10^{-4}$ & $0.0009\pm0.0003$ & $0.0030\pm0.0014$ & $0.0028\pm0.0023$ & fail & fail \\
Noise & $3\cdot 10^{-4}$ & $0.0009\pm0.0003$ & $0.0031\pm0.0014$ & $0.0016\pm0.0003$ & fail & fail \\
Noise & $10^{-3}$ & $0.0009\pm0.0004$ & $0.0034\pm0.0015$ & $0.0016\pm0.0004$ & fail & fail \\
Noise & $3\cdot 10^{-3}$ & $0.0008\pm0.0002$ & $0.0039\pm0.0015$ & $0.0016\pm0.0003$ & fail & fail \\
\hline
Outlier & $0$ & $0.0009\pm0.0003$ & $0.0035\pm0.0013$ & $0.0018\pm0.0005$ & fail & fail \\
Outlier & $10^{-3}$ & $0.0008\pm0.0001$ & $0.0039\pm0.0018$ & $0.0021\pm0.0013$ & fail & fail \\
Outlier & $3\cdot 10^{-3}$ & $0.0008\pm0.0002$ & $0.0032\pm0.0017$ & $0.0018\pm0.0002$ & fail & fail \\
Outlier & $10^{-2}$ & $0.0008\pm0.0002$ & $0.0036\pm0.0018$ & $0.0023\pm0.0009$ & fail & fail \\
\hline
\end{tabular}
\end{table}

The results show that RN-SLRA remains robust and consistently achieves the
smallest mean CPU time across both sweeps. VP-SLRA and Cadzow also converge,
but require more time, while STLN and Gradient-system fail on all tested
instances. This further illustrates the effectiveness of RN-SLRA on examples
with poor geometry.

\subsubsection{Large-scale inexact-projection experiment}
We finally compare RN-SLRA with iRN-SLRA when the rank-\(r\) projection is
computed only approximately. Here \(n\) and \(m\) denote the matrix dimensions
and \(r\) is the target rank. We use larger degenerate Hankel SLRA instances
\(H_c^{(n,m,3)}\in\mathbb R^{n\times m}\) generated from the repeated-pole
model \(\nu_i=\beta_0\lambda^i+\beta_1 i\lambda^i+\beta_2\mu^i\), where the
term \(i\lambda^i\) induces poor local geometry. We test
\((n,m,r)\in\{(120,160,3),(200,280,3),(350,500,3),(700,1000,3)\}\).

To keep the instances comparable, the poles are moved closer to one as the
matrix size increases and the coefficients are mildly rescaled. 
For these four sizes, the corresponding
\((\lambda,\mu,\beta_0,\beta_1,\beta_2)\) are
$(e^{-0.06},e^{-0.19},1.0,0.09,0.8)$,
$(e^{-0.035},e^{-0.12},1.0,0.06,$
$0.65)$,
$(e^{-0.022},e^{-0.075},1.0,0.045,0.55)$, and
$(e^{-0.012},e^{-0.045},1.0,0.025,0.45)$.

The initial matrix is \(H_0=H_c+\tau\Delta+\eta E_{\mathrm{out}}\), where
\(\Delta\) is random Hankel noise and \(E_{\mathrm{out}}\) is a Hankel outlier
supported on one anti-diagonal. We normalize the perturbations by
\(\alpha_c=\|H_c\|/\sqrt{mn}\), and set
\(\tau=5\times10^{-2}\alpha_c\), \(\eta=10^{-1}\alpha_c\). For both methods,
we tune \(c\in\{10^{-8},3\times10^{-8},10^{-7},3\times10^{-7},10^{-6}\}\) and
\(\rho\in\{0,0.5,1.0\}\). For iRN-SLRA, the rank-\(r\) projection is
approximated by randomized SVD \cite{halko2011finding}.
We use \(\varepsilon_{\rm r}\in\{0.1,0.25\}\)
as the relative tolerance in the randomized SVD routine, and
\(s\in\{1,2,3\}\) as the number of subspace iterations.
The stopping tolerance is \(10^{-7}\), and the maximum
number of outer iterations is \(40\).
\begin{table}[htbp]
\caption{Large-scale Hankel experiments: iteration counts, CPU times, and final rank residuals.}
\label{tab:Hankel-large-irn}
\centering
\small
\begin{tabular}{c|ccc|ccc}
\hline
\((n,m,r)\) & \multicolumn{3}{c|}{RN-SLRA}
& \multicolumn{3}{c}{iRN-SLRA} \\
& Iter. & Time (s) & Final residual & Iter. & Time (s) & Final residual \\
\hline
(120,160,3) & 6 & 0.8515 & $2.62{\times}10^{-10}$ & 4 & 0.4793 & $5.60{\times}10^{-10}$ \\
(200,280,3) & 7 & 2.8913 & $1.63{\times}10^{-10}$ & 4 & 1.3842 & $3.82{\times}10^{-10}$ \\
(350,500,3) & 5 & 9.3378 & $1.33{\times}10^{-10}$ & 4 & 6.8277 & $3.68{\times}10^{-13}$ \\
(700,1000,3) & 6 & 98.1399 & $2.33{\times}10^{-10}$ & 4 & 58.0918 & $5.74{\times}10^{-10}$ \\
\hline
\end{tabular}
\end{table}

Table~\ref{tab:Hankel-large-irn} shows that iRN-SLRA preserves the robustness
of RN-SLRA while reducing the running time. The gain becomes more pronounced
at larger scales; for \((700,1000,3)\), the time decreases from \(98.14\) to
\(58.09\) seconds. 

\section{Conclusion}

We proposed RN-SLRA, a regularized Newton-type method for local
manifold--affine intersection problems. \blue{Under intrinsic transversality,
equivalently clean intersection in the present $C^2$ manifold--affine setting,
RN-SLRA converges locally linearly for $\rho=0$ and with Q-order at least
$1+\rho$ for $0<\rho\le1$, including quadratic convergence for $\rho=1$.
For iRN-SLRA with $\sigma$-quasioptimal manifold projections, the same
higher-order rates hold for $0<\rho\le1$, while fixed regularization yields local linear
convergence when $q\sigma<1$. The experiments support the predicted quadratic
behavior on an analytically verified clean but nontransversal instance and
illustrate the computational efficiency of iRN-SLRA on larger Hankel problems.}


\bibliographystyle{siamplain}
\bibliography{references}

\appendix

\section{\texorpdfstring{\blue{Geometric Analysis and Numerical Diagnostics}}{Geometric Analysis and Numerical Diagnostics}}
\label{app:intrinsic transversality}

\subsection{Numerical Principle}

Let \(E\) be the affine constraint set and let \(D_r\) be the fixed-rank manifold of rank-\(r\) matrices. Around a point \(x^*\in E\cap D_r\), we sample nearby points
\(x\in E\) and \(y\in D_r\), and define
\(u=(x-y)/\|x-y\|_F\).
We then compute the intrinsic transversality proxy
\[
    \blue{\eta(x,y)
    =
    \max\Big\{
        \operatorname{dist}\big(u,N_y{D_r}^0\big),
        \operatorname{dist}\big(u,N_xE^0\big)
    \Big\}.}
\]
Since \(E\) is affine, \(\blue{N_xE^0=L^\perp}\) is independent of \(x\). A positive lower
bound on \(\eta(x,y)\) over nearby sampled pairs indicates that the displacement
direction cannot be simultaneously close to both normal geometries, which is
the separation expected from intrinsic transversality.
\blue{A positive minimum over finitely many sampled pairs is only a numerical
diagnostic of local separation; it is not a proof of intrinsic transversality.
In particular, isolated bad directions can be missed by random sampling.
Whenever exact intersection geometry is needed below, we verify it
analytically.}

\blue{For the synthetic experiments we also report singular-value or rank
diagnostics for the classical tangent-intersection system. Sampled values of
$\eta$, when used below, are only numerical diagnostics and are not invoked to
establish clean intersection.}

\subsection{\texorpdfstring{\blue{Exact Clean but Nontransversal Geometry of the Small-Scale Problem}}{Exact Clean but Nontransversal Geometry of the Small-Scale Problem}}

\blue{Parameterize the affine space using the raw directions from
Section~5.1 as
\(\Phi(z)=x^*+\sum_{i=1}^4 z_i\widetilde B_i\), where
\(z=(z_1,z_2,z_3,z_4)\).
Since Gram--Schmidt orthonormalization does not change their span, this
parameterizes the same affine space $E$. With the $2\times2$ block partition
\[
\Phi(z)=\begin{pmatrix}A(z)&B(z)\\C(z)&D(z)\end{pmatrix},
\]
we have
\[
A(z)=\begin{pmatrix}4+z_1&0.2z_4\\0&0.02\end{pmatrix},\quad
B(z)=\begin{pmatrix}z_2&0\\0&z_3\end{pmatrix},
\]
\[
C(z)=\begin{pmatrix}0.7z_2&0\\0&-0.6z_3\end{pmatrix},\quad
D(z)=\begin{pmatrix}z_2&z_4\\0&z_3\end{pmatrix}.
\]
For $z$ sufficiently close to $0$, $A(z)$ is invertible, and the
Schur-complement criterion gives
\[
\operatorname{rank}\Phi(z)=2
\quad\Longleftrightarrow\quad
F(z):=D(z)-C(z)A(z)^{-1}B(z)=0.
\]
A direct calculation yields
\[
F(z)=\begin{pmatrix}
z_2-\dfrac{0.7z_2^2}{4+z_1}&
z_4+\dfrac{7z_2z_3z_4}{4+z_1}\\[1ex]
0&z_3+30z_3^2
\end{pmatrix}.
\]
After shrinking the neighborhood if necessary, the factors
$1-0.7z_2/(4+z_1)$, $1+30z_3$, and
$1+7z_2z_3/(4+z_1)$ are nonzero. Hence $F(z)=0$ implies
$z_2=z_3=z_4=0$. Conversely, these three equalities make every entry of
$F(z)$ vanish, while $z_1$ remains free. Consequently, for a sufficiently
small neighborhood $U$ of $x^*$, we have 
\(
E\cap\mathcal D_2\cap U
=\{x^*+tE_{11}:|t|<\varepsilon\}.
\)
Thus the local intersection is a one-dimensional smooth embedded submanifold.}

\blue{Moreover, at every parameter $z=(t,0,0,0)$ representing a point on
this local intersection,
\[
DF(z)[h]=\begin{pmatrix}h_2&h_4\\0&h_3\end{pmatrix},
\qquad
\ker DF(z)=\operatorname{span}\{e_1\}.
\]
The four raw directions are linearly independent, so
$D\Phi(z):\mathbb R^4\to L$ is a linear isomorphism onto the direction space
$L=E^0$. Since the Schur-complement equations locally define
$\mathcal D_2$ within the affine chart $E$, the tangent space of their zero set
is carried by this isomorphism to the intersection of the underlying direction
spaces. Therefore,
\[
L\cap T_{\Phi(z)}\mathcal D_2^0
=D\Phi(z)\ker DF(z)
=\operatorname{span}\{E_{11}\}
=T_{\Phi(z)}(E\cap\mathcal D_2)^0.
\]
Hence $E$ and $\mathcal D_2$ intersect cleanly near $x^*$.}

\blue{The intersection is nevertheless not transversal. Indeed, $\dim E=4$,
while the rank-$2$ manifold in $\mathbb R^{4\times4}$ has dimension
\(
\dim\mathcal D_2=2(4+4-2)=12.
\)
Transversality would therefore require
\(
\dim\bigl(L\cap T_{x^*}\mathcal D_2^0\bigr)=4+12-16=0.
\)
The clean-intersection calculation above instead gives
\(
L\cap T_{x^*}\mathcal D_2^0=\operatorname{span}\{E_{11}\},
\)
which is one-dimensional. Thus classical transversality fails. In agreement
with this exact calculation, the singular values of the classical
tangent-intersection matrix are
\(
\sigma(A(x^*))=(0.98058,0.65094,0.63372,0).
\)
The explicit local description also allows the local solution-set distance
used in Section~5.1 to be computed as the Frobenius distance to the affine line
$x^*+\operatorname{span}\{E_{11}\}$, provided the iterates remain in the
neighborhood above.}

\subsection{Large-Scale Synthetic SLRA Examples}

For the nine large-scale synthetic SLRA examples, we report a rank bound for
the classical transversality matrix and the minimum sampled intrinsic transversality proxy.
The proxy is computed over 24 local samples with sampling radius \(0.01\).

\[
\begin{array}{ccccc}
\hline
 (m,n) & p & r & \text{rank bound} & \min \eta\\
\hline
 (100,100) & 50 & 10 & 2 & 0.9998316\\
 (100,100) & 50 & 80 & 2 & 0.9997508\\
 (100,100) & 500 & 10 & 2 & 0.9995969\\
 (100,100) & 500 & 80 & 2 & 0.9995995\\
 (1000,1000) & 50 & 10 & 2 & 0.9999997\\
 (1000,1000) & 50 & 100 & 2 & 0.9999983\\
 (1000,1000) & 500 & 10 & 2 & 0.9999981\\
 (1000,1000) & 500 & 100 & 2 & 0.9999890\\
 (1000,100) & 500 & 10 & 2 & 0.9999526\\
\hline
\end{array}
\]

In all nine cases, the sampled intrinsic proxy is close to one. In contrast,
the recorded transversality rank bound is \(2\), much smaller than the number
of affine constraints \(p\). This is consistent with the construction: the
examples are rank deficient in the classical tangent-intersection sense, while
retaining a strong sampled intrinsic separation.

\subsection{Hankel SLRA Experiments}

We also computed the intrinsic transversality proxy for the small Hankel
experiments. \blue{For each case, 200 sampled pairs were used at sampling radii
\(0.01\), \(0.02\), and \(0.05\).} The table reports the smallest sampled value of
\(\eta\) over these radii, together with the range of median values. The column
\(\sigma_{\min}(A)\) gives the smallest singular value of the tangent-restriction
matrix when it was explicitly formed.

\[
\begin{array}{lccccc}
\hline
\text{experiment} & (m,n) & r & \sigma_{\min}(A) & \min \eta & \operatorname{median}\eta\text{ range}\\
\hline
\text{single example} & (7,5) & 4
& 4.5647\cdot 10^{-1} & 0.8628 & [0.9788,0.9800]\\
\text{noise sweep} & (7,5) & 4
& 4.5575\cdot 10^{-1} & 0.8234 & [0.9783,0.9788]\\
\text{noise sweep} & (7,5) & 4
& 4.5647\cdot 10^{-1} & 0.8710 & [0.9788,0.9802]\\
\text{noise sweep} & (8,5) & 2
& 4.8981\cdot 10^{-17} & 0.7718 & [0.8812,0.8838]\\
\text{outlier sweep} & (7,5) & 4
& 4.5575\cdot 10^{-1} & 0.8864 & [0.9803,0.9815]\\
\text{outlier sweep} & (7,5) & 4
& 4.5647\cdot 10^{-1} & 0.8391 & [0.9778,0.9799]\\
\text{outlier sweep} & (8,5) & 2
& 4.8981\cdot 10^{-17} & 0.7627 & [0.8761,0.8792]\\
\hline
\end{array}
\]

For all small Hankel cases, the sampled intrinsic transversality proxy remains bounded away
from zero. Some cases have a nearly singular tangent-restriction matrix, while
their intrinsic transversality proxy remains positive. This is consistent with the pattern
observed in the synthetic SLRA examples.

\subsection{\texorpdfstring{Large-Scale Hankel Intrinsic Transversality Diagnostics}{Large-Scale Hankel intrinsic transversality Diagnostics}}

For the large-scale Hankel benchmark used in
Table~\ref{tab:Hankel-large-irn}, the tangent-restriction 
matrix was not explicitly formed. The intrinsic transversality proxy was computed directly
from sampled normal-cone distances, using 24 sampled pairs at radius \(0.01\)
for each case.

\[
\begin{array}{ccccc}
\hline
\blue{(n,m)} & r & \text{samples} & \min\eta & \operatorname{median}\eta\\
\hline
\blue{(120,160)} & 3 & 24 & 0.9752 & 0.9786\\
\blue{(200,280)} & 3 & 24 & 0.9852 & 0.9879\\
\blue{(350,500)} & 3 & 24 & 0.9920 & 0.9927\\
\blue{(700,1000)} & 3 & 24 & 0.9961 & 0.9964\\
\hline
\end{array}
\]
All four large-scale Hankel cases have sampled \(\eta\)-values close to one,
providing further numerical evidence of intrinsic separation.

\section{Retraction induced by RN-SLRA in the case \texorpdfstring{\(\rho=0\)}{rho=0}}
\label{app:rho-zero-retraction}

\blue{This appendix further studies the fixed-regularization case $\rho=0$, for which
Section~\ref{sec:local-convergence} establishes local linear convergence. We
justify the retraction interpretation in
Remark~\ref{rmk:retraction-clean-intersection} by proving smoothness of the
limit map and its genuine and second-order retraction properties under the
stated additional regularity assumptions.}

Let \(\Phi_c:E\to E\) denote the one-step RN-SLRA map with
\(\mu_k\equiv c\). For \(z\in E\) close to \(X\), set
\(y=P_M(z)\), \(Q_y:=P_{\blue{N_yM^0}}\), and define
\[
        \Phi_c(z)
        :=
        \arg\min_{u\in E}
        \left\{
        \frac12\|u-y\|^2
        +
        \frac{1}{2c}\|Q_y(u-y)\|^2
        \right\}.
\]
We first show that, under clean intersection and suitable smoothness, the
limiting map of \(\Phi_c\) defines a genuine retraction on \(X\). We then show
that one additional order of smoothness yields a second-order retraction.

\begin{theorem}[Retraction for RN-SLRA when \(\rho=0\)]
\label{thm:rho-zero-retraction} Assume that $M$ is locally of class $C^3$ near $\bar x$ and that
$E$ intersects $M$ cleanly at $\bar x$.  Then, for all \(z\in E\) sufficiently close to \(X\), the
limit
\[
        \Psi_c(z):=\lim_{k\to\infty}\Phi_c^k(z)
\]
exists locally and \(\Psi_c\) is \(C^1\). Consequently,
\[
        R_x^c(\xi):=\Psi_c(x+\xi),
        \qquad x\in X,\quad \blue{\xi\in T_xX^0},
\]
defines a \(C^1\) retraction on \(X\).
\end{theorem}

\begin{proof}
Since $M$ is $C^3$,  
the fixed-regularization one-step map \(\Phi_c\) is of class \(C^2\) in a
neighborhood of \(X\).
For \(y\) close to \(X\), write \(p=P_E(y)\). Since every point in \(E\)
can be written as \(p+\ell\) with \(\ell\in L\), the subproblem defining
\(\Phi_c(z)\) is equivalent to minimizing over \(\ell\in L\). Its first-order
optimality condition is
\[
        \ell+\frac1c P_LQ_y(p-y+\ell)=0 .
\]
Equivalently,
\[
        (cI_L+A_y)\ell=-P_LQ_y(p-y),
        \qquad
        A_y:=P_LQ_y|_L .
\]
Since \(A_y\) is self-adjoint positive semidefinite on \(L\) and \(c>0\),
the operator \(cI_L+A_y\) is uniformly invertible. Hence \(\Phi_c\) is locally
well defined and inherits the assumed smoothness.

For \(x\in X\), we have \(P_M(x)=x\), \(P_E(x)=x\), and hence
\(\Phi_c(x)=x\). Thus \(X\) is a fixed-point manifold of \(\Phi_c\).
We now compute the derivative of \(\Phi_c\) at \(x\in X\). Let \(h\in L\).
Since \(\mathrm{D}P_M(x)h=P_{\blue{T_xM^0}}h\), differentiating the optimality system
at \(x\) gives
\[
        (cI_L+A_x)\ell'
        =
        -P_LQ_x\big(P_LP_{\blue{T_xM^0}}h-P_{\blue{T_xM^0}}h\big),
        \qquad
        A_x:=P_LQ_x|_L .
\]
Using \(P_{\blue{T_xM^0}}=I-Q_x\) and \(h\in L\), we have
\(P_LP_{\blue{T_xM^0}}h=h-A_xh\). Hence the right-hand side equals
\(-A_x(I_L-A_x)h\), and therefore
\[
        \ell'=-(cI_L+A_x)^{-1}A_x(I_L-A_x)h .
\]
Since \(\mathrm{D}\Phi_c(x)h=P_LP_{\blue{T_xM^0}}h+\ell'\), we obtain
\begin{equation}
\label{eq:rho0-DPhi}
        \mathrm{D}\Phi_c(x)h
        =
        \left(I_L+c^{-1}A_x\right)^{-1}(I_L-A_x)h .
\end{equation}

\blue{By Lemma~\ref{lem:effective-coercivity}, $A_x$ is self-adjoint and
uniformly positive definite on $W_x=L\ominus \blue{T_xX^0}$, with
$\ker A_x=\blue{T_xX^0}$.}
Thus \(\mathrm{D}\Phi_c(x)h=h\) for \(h\in \blue{T_xX^0}\). If \(h\in W_x\) is an
eigenvector of \(A_x\) with eigenvalue \(\lambda\in[\blue{\lambda_*},1]\), then
\[
        \mathrm{D}\Phi_c(x)h
        =
        \frac{1-\lambda}{1+\lambda/c}h .
\]
Consequently, the derivative is the identity along \(T_xX\) and is uniformly
contractive on \(W_x\). Therefore \(X\) is a normally attracting fixed-point
manifold of \(\Phi_c\).

We next prove the \(C^1\) regularity of the limiting map. We work in a
local tubular coordinate system in \(E\) around \(X\). Namely, for \(z\in E\)
close to \(X\), write
\[
        z=s+w,\qquad s=P_X(z),\qquad w\in W_s:=L\ominus \blue{T_sX^0} .
\]
Using a smooth local frame of the bundle \(W_s\), we identify the spaces
\(W_s\) with a fixed Euclidean space. In these coordinates, write
\[
        \Phi_c(s,w)=\big(s+a(s,w),\,B_s w+b(s,w)\big).
\]
Since every point of \(X\) is fixed by \(\Phi_c\), and since the derivative
of \(\Phi_c\) is the identity in the tangent direction and has no first-order
tangential response in the normal direction, we have
\[
        a(s,0)=0,\qquad b(s,0)=0,\qquad
        D_w a(s,0)=0,\qquad D_w b(s,0)=0.
\]
Moreover, by the uniform contraction on \(W_s\), after shrinking the
neighborhood if necessary,
\[
        \|B_s\|\le q<1
\]
uniformly for \(s\in X\) close to \(\bar x\). Since \(\Phi_c\) is \(C^2\), it follows that
\begin{equation}
\label{eq:rho0-ab-quadratic}
        a(s,w)=O(\|w\|^2),
        \qquad
        b(s,w)=O(\|w\|^2).
\end{equation}
Moreover,
\begin{equation}
\label{eq:rho0-first-derivative-estimates}
\begin{aligned}
        D_s a(s,w)&=O(\|w\|),&
        D_w a(s,w)&=O(\|w\|),\\
        D_s b(s,w)&=O(\|w\|),&
        D_w b(s,w)&=O(\|w\|).
\end{aligned}
\end{equation}
All estimates are local and uniform.

Let
\[
        (s_{k+1},w_{k+1})=\Phi_c(s_k,w_k),
        \qquad
        (s_0,w_0)=(s,w).
\]
The normal component satisfies
\[
        w_{k+1}=B_{s_k}w_k+b(s_k,w_k).
\]
Choose \(\theta\in(q,1)\). By \eqref{eq:rho0-ab-quadratic}, after shrinking
the neighborhood if necessary, the nonlinear term \(b(s,w)\) is small enough
so that
\begin{equation}
\label{eq:rho0-normal-decay}
        \|w_k\|\le C\theta^k\|w_0\|.
\end{equation}
Moreover,
\[
        s_{k+1}-s_k=a(s_k,w_k),
\]
and hence, by \eqref{eq:rho0-ab-quadratic} and
\eqref{eq:rho0-normal-decay},
\[
        \sum_{k=0}^{\infty}\|s_{k+1}-s_k\|
        \le
        C\sum_{k=0}^{\infty}\|w_k\|^2
        \le
        C\|w_0\|^2.
\]
Therefore \(s_k\) converges to a limit \(s_\infty\), \(w_k\to0\), and
\begin{equation}
\label{eq:rho0-limit-representation}
        \Psi_c(s,w)=(s_\infty,0),
        \qquad
        s_\infty=s+\sum_{k=0}^{\infty}a(s_k,w_k).
\end{equation}

It remains to justify the \(C^1\) dependence on the initial point. Differentiating
the recursions and using \eqref{eq:rho0-first-derivative-estimates},
\eqref{eq:rho0-normal-decay}, and the uniform contraction in the
\(w\)-equation gives, by a standard induction on the differentiated recursions,
\begin{equation}
\label{eq:rho0-first-derivative-bounds}
        \|D_z w_k\|\le C\theta^k,
        \qquad
        \sup_k\|D_z s_k\|\le C .
\end{equation}
By the chain rule and \eqref{eq:rho0-first-derivative-estimates},
\[
\begin{aligned}
\|D_z a(s_k,w_k)\|
&\le
C\|w_k\|\big(\|D_zs_k\|+\|D_zw_k\|\big)  \\
&\le C'\theta^k .
\end{aligned}
\]
Thus the derivative series
\[
        D_zs_0+\sum_{k\ge0}D_z a(s_k,w_k)
\]
converges uniformly. It follows from \eqref{eq:rho0-limit-representation} that
\(s_\infty\) is \(C^1\), and therefore
\(\Psi_c(s,w)=(s_\infty(s,w),0)\) is \(C^1\).

Finally, define, for \(x\in X\) close to \(\bar x\) and
\(\xi\in \blue{T_xX^0}\) sufficiently small,
\[
        R_x^c(\xi):=\Psi_c(x+\xi).
\]
\blue{Let \(TX^0
:=
\bigcup_{x\in X}\{x\}\times T_xX^0.\)} Since \(\Psi_c\) is \(C^1\), \(R^c:\blue{TX^0}\to X\) is \(C^1\). Also,
\(\Psi_c(x)=x\) for every \(x\in X\). Hence, for any
\(\xi\in T_xX\), taking a \(C^1\) curve \(\gamma(t)\subset X\) with
\(\gamma(0)=x\) and \(\gamma'(0)=\xi\), we have
\(\Psi_c(\gamma(t))=\gamma(t)\), and therefore
\(\mathrm{D}\Psi_c(x)\xi=\xi\). Consequently,
\[
        R_x^c(0_x)=x,
        \qquad
        \mathrm{D}R_x^c(0_x)=\mathrm{id}_{\blue{T_xX^0}}.
\]
Thus \(R^c\) is a \(C^1\) retraction on \(X\).
\end{proof}

\begin{theorem}[Second-order retraction for RN-SLRA when \(\rho=0\)]
\label{thm:rho-zero-second-order}
Assume the hypotheses of Theorem~\ref{thm:rho-zero-retraction}. Assume in
addition that  \(M\) is
locally \(C^4\) near \(\bar x\). Then the limiting map \(\Psi_c\) is \(C^2\).
Consequently,
\[
        R_x^c(\xi):=\Psi_c(x+\xi),
        \qquad x\in X,\quad \xi\in \blue{T_xX^0},
\]
defines a second-order retraction on \(X\).
\end{theorem}

\begin{proof}
We use the notation and local coordinates from the proof of
Theorem~\ref{thm:rho-zero-retraction}. 
Since now $\Phi_c$ is $C^3$, the functions $a$ and $b$ satisfy the following
strengthened local estimates. Indeed, the fixed-point property on $X$ gives
$a(s,0)=b(s,0)=0$, and the normal-form construction gives
$D_w a(s,0)=D_w b(s,0)=0$. Hence, by Taylor's formula in the normal variable $w$,
uniformly for $s$ near $\bar x$,
\[
a(s,w)=O(\|w\|^2),\qquad b(s,w)=O(\|w\|^2),
\]
and, differentiating this Taylor representation with respect to $s$ and $w$,
\begin{equation}
\label{eq:rho0-strong-first-derivative-estimates}
\begin{aligned}
        D_s a(s,w)&=O(\|w\|^2),&
        D_w a(s,w)&=O(\|w\|),\\
        D_s b(s,w)&=O(\|w\|^2),&
        D_w b(s,w)&=O(\|w\|),
\end{aligned}
\end{equation}
and, for the second derivatives,
\begin{equation}
\label{eq:rho0-second-derivative-estimates}
\begin{aligned}
D_s^2a(s,w)&=O(\|w\|),&
D_sD_w a(s,w)&=O(\|w\|),&
D_w^2a(s,w)&=O(1),\\
D_s^2b(s,w)&=O(\|w\|),&
D_sD_w b(s,w)&=O(\|w\|),&
D_w^2b(s,w)&=O(1).
\end{aligned}
\end{equation}
All estimates are local and uniform. The weaker \(O(\|w\|)\) estimate for
\(D_s^2a\) and \(D_s^2b\) is sufficient for the argument below.

We already know from the proof of Theorem~\ref{thm:rho-zero-retraction} that
\begin{equation}
\label{eq:rho0-basic-bounds}
        \|w_k\|\le C\theta^k\|w_0\|,
        \qquad
        \|D_z w_k\|\le C\theta^k,
        \qquad
        \sup_k\|D_zs_k\|\le C.
\end{equation}

\blue{We now control the second derivatives. Set
\[
S_k:=D_zs_k,\qquad
W_k:=D_zw_k,\qquad
H_k:=D_z^2s_k,\qquad
K_k:=D_z^2w_k.
\]

We first consider the normal recursion
\(
w_{k+1}=B_{s_k}w_k+b(s_k,w_k).
\)
Let
\(
F(s,w):=B_sw+b(s,w).
\)
By \eqref{eq:rho0-strong-first-derivative-estimates} and
\eqref{eq:rho0-second-derivative-estimates},
\[
D_wF(s,w)=B_s+O(\|w\|),\qquad
D_sF(s,w)=O(\|w\|),
\]
while
\[
D_s^2F(s,w)=O(\|w\|),\qquad
D_sD_wF(s,w)=O(1),\qquad
D_w^2F(s,w)=O(1).
\]
Recall that \(\|B_s\|\le q<\theta<1\) uniformly. After shrinking the
neighborhood if necessary, choose \(\widehat q\) with
\(
q<\widehat q<\theta
\)
such that
\(
\|D_wF(s,w)\|\le \widehat q
\)
throughout the neighborhood under consideration.

Differentiating the normal recursion twice and using
\eqref{eq:rho0-basic-bounds}, we therefore obtain
\begin{equation}
\label{eq:rho0-second-w-recursion}
\|K_{k+1}\|
\le
\widehat q\,\|K_k\|
+
C\theta^k\bigl(1+\|H_k\|\bigr).
\end{equation}
Indeed, the term involving \(K_k\) is
\(D_wF(s_k,w_k)K_k\). The term involving \(H_k\) has coefficient
\(D_sF(s_k,w_k)=O(\|w_k\|)=O(\theta^k)\), while the quadratic
terms in \(S_k\) and \(W_k\) are bounded by \(C\theta^k\) in view of
\eqref{eq:rho0-basic-bounds}.

For the tangential recursion
\(
s_{k+1}=s_k+a(s_k,w_k),
\)
the second-order chain rule gives
\[
\begin{aligned}
H_{k+1}
={}&H_k
+D_sa(s_k,w_k)H_k
+D_wa(s_k,w_k)K_k+D_s^2a(s_k,w_k)[S_k,S_k]\\
&+2D_sD_wa(s_k,w_k)[S_k,W_k]+D_w^2a(s_k,w_k)[W_k,W_k].
\end{aligned}
\]
Using
\eqref{eq:rho0-strong-first-derivative-estimates}--%
\eqref{eq:rho0-basic-bounds}, we obtain
\begin{equation}
\label{eq:rho0-second-s-recursion}
\|H_{k+1}\|
\le
\bigl(1+C\theta^{2k}\bigr)\|H_k\|
+
C\theta^k\|K_k\|
+
C\theta^k.
\end{equation}
Here the last term is only of order \(O(\theta^k)\): indeed,
\[
D_s^2a(s_k,w_k)[S_k,S_k]=O(\theta^k),
\]
because \(D_s^2a(s_k,w_k)=O(\|w_k\|)=O(\theta^k)\) whereas
\(\|S_k\|\) is uniformly bounded.

Set
\(
M_k:=1+\max_{0\le j\le k}\|H_j\|.
\)
Iterating \eqref{eq:rho0-second-w-recursion}, and using
\(\widehat q<\theta\) and the monotonicity of \(M_k\), gives
\[
\begin{aligned}
\|K_k\|
\le
\widehat q^{\,k}\|K_0\|
+
C\sum_{j=0}^{k-1}
\widehat q^{\,k-1-j}\theta^j M_j\le
C\theta^k M_k.
\end{aligned}
\]
Substituting this estimate into
\eqref{eq:rho0-second-s-recursion}, we obtain
\[
\begin{aligned}
1+\|H_{k+1}\|
\le
M_k
+
C\theta^{2k}M_k
+
C\theta^k\le
\bigl(1+C\theta^k\bigr)M_k.
\end{aligned}
\]
Consequently,
\(
M_{k+1}\le \bigl(1+C\theta^k\bigr)M_k.
\)
Since
\(
\sum_{k=0}^{\infty}\theta^k<\infty,
\)
it follows that
\[
M_k
\le
M_0\prod_{j=0}^{k-1}(1+C\theta^j)
\le
M_0\exp\left(C\sum_{j=0}^{\infty}\theta^j\right)
\le C.
\]
Returning to the estimate for \(K_k\), we conclude that
\begin{equation}
\label{eq:rho0-second-derivative-bounds}
\|D_z^2w_k\|\le C\theta^k,
\qquad
\sup_k\|D_z^2s_k\|\le C.
\end{equation}}

We next show that the second derivative series of \(s_\infty\) converges
uniformly. By the chain rule,
\[
\begin{aligned}
D_z^2(a(s_k,w_k))
&=
D^2a(s_k,w_k)
        [D_z(s_k,w_k),D_z(s_k,w_k)]  \\
&\quad
+
Da(s_k,w_k)D_z^2(s_k,w_k).
\end{aligned}
\]
Using \eqref{eq:rho0-strong-first-derivative-estimates}--\eqref{eq:rho0-second-derivative-bounds},
each term on the right-hand side is bounded by \(C\theta^k\). Hence
\[
        \|D_z^2(a(s_k,w_k))\|\le C\theta^k,
\]
and therefore
\[
        \sum_{k=0}^{\infty}D_z^2(a(s_k,w_k))
\]
converges uniformly. The first derivative series already converges uniformly by
Theorem~\ref{thm:rho-zero-retraction}. Thus
\[
        s_\infty(s,w)
        =
        s+\sum_{k=0}^{\infty}a(s_k,w_k)
\]
is \(C^2\), and consequently \(\Psi_c(s,w)=(s_\infty(s,w),0)\) is \(C^2\).

It remains to verify the second-order retraction property. From
\eqref{eq:rho0-ab-quadratic}, \eqref{eq:rho0-normal-decay}, and
\eqref{eq:rho0-limit-representation}, we have
\[
        s_\infty-s
        =
        \sum_{k=0}^{\infty}a(s_k,w_k)
        =
        O(\|w_0\|^2).
\]
Since \(s=P_X(z)\) and \(\|w_0\|=\operatorname{dist}(z,X)\), this gives
\begin{equation}
\label{eq:rho0-projection-second-order}
        \Psi_c(z)=P_X(z)+O(\operatorname{dist}(z,X)^2).
\end{equation}
Now take \(z=x+\xi\) with \(x\in X\) and \(\xi\in \blue{T_xX^0}\). Since the intersection is
clean and $M$ is locally $C^4$, the set $X=M\cap E$
is a $C^4$ embedded submanifold near $\bar x$. Hence
\[
        \operatorname{dist}(x+\xi,X)=O(\|\xi\|^2),
\]
and the metric projection retraction satisfies the standard expansion
\[
        P_X(x+\xi)
        =
        \operatorname{Exp}_x(\xi)+O(\|\xi\|^3).
\]
Combining this with \eqref{eq:rho0-projection-second-order}, we obtain
\[
        R_x^c(\xi)
        =
        \Psi_c(x+\xi)
        =
        P_X(x+\xi)+O(\|\xi\|^4)
        =
        \operatorname{Exp}_x(\xi)+O(\|\xi\|^3).
\]
Since \(R^c\) is \(C^2\), this proves that \(R^c\) is a second-order retraction
on \(X\).
\end{proof}

\begin{remark}[Weaker smoothness assumptions]
	The smoothness assumptions in Theorems~\ref{thm:rho-zero-retraction} and \ref{thm:rho-zero-second-order} can be relaxed.
	Following the regularity analysis in \cite{chen2026retractionsalternatingprojections},
	one may replace the $C^3$ and $C^4$ assumptions on $M$ by
	$C^{2,1}$ and $C^{3,1}$, respectively, where $C^{p,1}$ denotes
	$C^p$ smoothness with locally Lipschitz $p$th derivative.
	The corresponding proofs follow the same differentiated-iteration
	argument, with the Lipschitz control providing the summability needed
	for the derivatives of the limiting map.
	We use the stronger $C^3$ and $C^4$ assumptions here only to keep the
	statement and regularity argument simpler.
\end{remark}
\end{document}